\documentclass[11pt]{article}
\usepackage{sectsty}

\usepackage{graphicx}

\usepackage{amsmath,amssymb}
\usepackage{bm}
\usepackage{enumerate}
\usepackage{amsthm}
\usepackage[all,2cell]{xy}
\usepackage{mathrsfs}
\usepackage{mathtools}

\mathtoolsset{showonlyrefs}

\usepackage{tikz-cd}
\usepackage{xcolor}
\DeclareUnicodeCharacter{0301}{*************************************}

\usepackage[colorlinks=true]{hyperref}
\hypersetup{urlcolor=blue, citecolor=red, linkcolor=blue}

\usetikzlibrary{cd}

\usepackage[a4paper, left=25mm, right=25mm, top=30mm, bottom=30mm]{geometry}

\usepackage{marginnote}

\usepackage{imakeidx}
\makeindex[intoc]

\usepackage[nottoc]{tocbibind}
\usepackage[colorinlistoftodos]{todonotes} 

\def\baselinestretch{1.1}
\newtheorem{definition}{Definition}[section]

\newtheorem{theorem}[definition]{Theorem}
\newtheorem{proposition}[definition]{Proposition}
\newtheorem{corollary}[definition]{Corollary}
\newtheorem{remarkth}[definition]{Remark}
\newtheorem{example}[definition]{Example}
\newenvironment{remark}{\begin{remarkth}\upshape}{\hfill$\diamond$\end{remarkth}}
\newcommand{\parder}[2]{\frac{\partial #1}{\partial #2}}

\newcommand{\tparder}[2]{\partial #1/\partial #2}
\newcommand{\parderr}[3]{\frac{\partial^2 #1}{\partial #2\partial #3}}

\newcommand{\lcf}{\lbrack\! \lbrack}
\newcommand{\rcf}{\rbrack\! \rbrack}

\renewcommand{\d}{\mathrm{d}^\circ}

\def\beq{\begin{equation}}
\def\eeq{\end{equation}}
\def\bea{\begin{eqnarray}}
\def\eea{\end{eqnarray}}
\def\beann{\begin{eqnarray*}}
\def\eeann{\end{eqnarray*}}
\def\ben{\begin{enumerate}}
\def\een{\end{enumerate}}
\def\bit{\begin{itemize}}
\def\eit{\end{itemize}}
\newcommand{\al}{\mathfrak{g}}

\newcommand{\br}[2]{\lbrack\!\lbrack#1,#2\rbrack\!\rbrack}

\def\derpar#1#2{\frac{\partial{#1}}{\partial{#2}}}

\newcommand\restr[2]{{
  \left.\kern-\nulldelimiterspace 
  #1 
  \right|_{#2} 
}}

\newcommand{\R}{\mathbb{R}}
\newcommand{\C}{\mathcal{C}}
\renewcommand{\d}{\mathrm{d}}
\newcommand{\Id}{\mathrm{Id}}

\newcommand{\E}{\mathcal{E}}

\newcommand{\vf}{\X}

\newcommand{\Cinfty}{\mathscr{C}^\infty}

\newcommand{\X}{\mathfrak{X}}
\newcommand{\Rb}{\mathcal{R}}

\def\d{\mathrm{d}}

\def\r{\R}

\def\rq{ \r   \times Q}

\def\er{E\times\r}
\def\ers{E^* \times \r}

\let\ds\displaystyle

 \def\bcr{\begin{color}{red}}
\def\bcb{\begin{color}{blue}}
\def\enc{\end{color}}

\def\Lec{\mathop{Leg}} 
\def\derpar#1#2{\displaystyle\frac{\partial{#1}}{\partial{#2}}}
\def\derpars#1#2#3{\displaystyle\frac{\partial^2{#1}}{\partial{#2}{\partial
{#3}}}}
\def\vf{\mathfrak{X}}
\def\d{{\rm d}}
\def\Cinfty{{\rm C}^\infty} 

\def\tol{\longrightarrow}

\def\tec{\mathcal{T}^E(E \times \r)}

\def\tech{\mathcal{T}^E(E^*  \times \r)}

\def\qed{\ifvmode\Realemovelastskip\fi
{\unskip\nobreak\hfil\penalty50\hbox{}\nobreak\hfil \hbox{\vrule
height1.2ex width1.2ex}\parfillskip=0pt \finalhyphendemerits=0
\par\smallskip}}

\def\qed{\ifvmode\removelastskip\fi
{\unskip\nobreak\hfil\penalty50\hbox{}\nobreak\hfil \hbox{\vrule
height1.2ex width1.2ex}\parfillskip=0pt \finalhyphendemerits=0
\par\smallskip}}

\title{\sf
Contact formalism for dissipative mechanical systems \\on Lie algebroids
}

\author{\sffamily 
Alexandre Anahory Simoes$^1$,\thanks{alexandre.anahory@ie.edu\quad ORCID: 0000-0003-4644-876X}
Leonardo Colombo$^2$,\thanks{leonardo.colombo@csic.es\quad ORCID:0000-0001-6493-6113 }
\\[0.1ex]
\sffamily 
Manuel de Le\'on$^{3,4}$, \thanks{mdeleon@usc.es\quad ORCID: 0000-0002-8028-2348}\,  
Modesto Salgado$^{5,6}$\thanks{modesto.salgado@usc.es\quad ORCID: 0000-0003-3982-1845}\,  
and 
Silvia Souto$^{5,6}$\thanks{silviasouto.perez@usc.es\quad ORCID: 0000-0003-0755-1211}
\\[1ex]
\\[0.1ex]
\normalsize\itshape\sffamily 
$^1$ IE School of Science and Technology, Madrid, Spain
\\[0.1ex]
\normalsize\itshape\sffamily 
$^2$Centre for Automation and Robotics (CSIC-UPM), Arganda del Rey, Spain
\\[0.1ex]
\normalsize\itshape\sffamily 
$^3$Instituto de Ciencias Matemáticas (CSIC), Madrid, Spain
\\[0.1ex]
\normalsize\itshape\sffamily 
$^4$Real Academia de Ciencias, Madrid, Spain
\\[0.1ex]
\normalsize\itshape\sffamily 
$^5$Departamento de Matemáticas, Facultade de Matemáticas,
\\[0.1ex]
\normalsize\itshape\sffamily 
Universidade de Santiago de Compostela, Spain
\\[0.1ex]
\normalsize\itshape\sffamily
$^6$Centro de Investigación y Tecnología Matemática de Galicia (CITMAga), Spain
}

\date{\sffamily   2022}

\begin{document}

\date{\today}



\maketitle

\begin{abstract}
 In this paper, we introduce a geometric description of  contact Lagrangian and Hamiltonian systems on Lie algebroids in the framework of contact geometry, using the theory of prolongations. We discuss the relation between Lagrangian and Hamiltonian settings through a convenient notion of Legendre transformation. We also discuss the Hamilton-Jacobi problem in this framework and introduce the notion of a Legendrian Lie subalgebroid of a contact Lie algebroid.
\end{abstract}


\noindent\textbf{Keywords:} 
Contact geometry, dissipative systems, Lie algebroids, Herglotz equations.\\

\noindent\textbf{MSC\,2020 codes:} 37J55, 53D10, 37C79, 37J37, 70H03, 70H05, 70H20


\pagestyle{myheadings}
\markright{\small\sf 
{\it A. Anahory, L. Colombo, M. de Le\'on, M. Salgado \textit{\&} S. Souto} ---
Contact formalism on Lie algebroids}

{\setcounter{tocdepth}{2}
\def\baselinestretch{1}
\small
\def\addvspace#1{\vskip 1pt}
\parskip 0pt plus 0.1mm
}

\section{Introduction}\label{intro}

The study of contact Hamiltonian systems has been experiencing enormous interest in recent years, due to their applications in fields such as thermodynamics, cosmology and neuroscience, to name but a few \cite{ALLM-2020, chossat, GB-2019, MNSS-2021, petitot} (see also \cite{BCT-2017, Brave-2017, CCM-2018, CdLL22, LGL-2022, LGLMR-2021, LJL-2020, LL-19, ML1-2019, ML-2020, MLMR-2021, LCdL-ACC} and the references therein). Their key properties lie in the fact that they model dissipative systems, as opposed to symplectic Hamiltonian systems, which serve as conservative models. Although familiar in thermodynamics (the geometric model is a contact manifold and equilibrium states are interpreted as Legendrian submanifolds), they were mostly used in their Hamiltonian formalism. However, the recovery of Herglotz's variational principle \cite{Her1930} (a generalisation of Hamilton's principle) has allowed the development of a Lagrangian formalism, which corresponds to what in physics are called action-dependent Lagrangians.

This interest has led to the extension of Lagrangian and Hamiltonian contact formalisms to other geometric contexts, such as Lie algebroids. This extension, already known in the case of symplectic Lagrangian and Hamiltonian systems, is not a mere mathematical formalism, but proves to be very useful for treating, for example, systems with symmetries, where the reduced system is no longer defined on a tangent or cotangent bundles but on a quotient of them. Indeed, the Lie algebroid context is a unifying concept (see \cite{AB-Med, C-JGM, CLMMM-2006, LMM-2005, LMSV-09, LS-2012, MV-10, Eduardoho, MC-2008, Tom, SMM-2002(2)}); the goal is to develop the program proposed by A. Weinstein in the early 1990s \cite{Weins-1996}.

There are two ways to extend these contact formalisms to algebroids. One is to extend the canonical Jacobi structure on $T^*Q \times \r$ to the case of $E^* \times \r$, where $E^*$ is the dual vector bundle of an algebroid $E$ over the configuration space $Q$. We have developed this approach in \cite{ACLSS-2023}. The advantage of this method is its simplicity, and the disadvantage is that we do not recover a contact structure, but a Jacobi structure (we also do not get a direct Lagrangian formulation).

The second method is to use the notion of prolongation of a Lie algebroid \cite{LMM-2005, Mack-1987, Mart-2001, Mart-2001(2), Mart-2005}. This has allowed us to define contact Lie algebroids, using the differential naturally associated to it, and therefore also the concept of a Legendrian Lie subalgebroid. We obtain the Euler-Lagrange and Hamilton equations, and relate them by a Legendre transformation. As in the case of usual contact Hamiltonian systems, associated to a Hamiltonian function we have an evolution section in addition to the Hamiltonian section. As an application of the results, we obtain the Hamilton-Jacobi equation for both the Hamiltonian section and the evolution section.

The organization of the paper is as follows. In section \ref{sec:contact-geometry}, we  recall some basic elements from  contact geometry. In section \ref{algebroids}, we remember some basic facts about Lie algebroids and the differential geometric aspects associated to them (see \cite{Mack-1995} for instance) . In this section, we also describe a particular example of a Lie algebroid, called the {\it prolongation of a Lie algebroid over a fibration}. This Lie algebroid will be necessary for further developments. In sections \ref{cftla} and \ref{Ha}, the contact formalism is extended to the setting of Lie algebroids; indeed, section \ref{cftla} describes the Lagrangian approach and section \ref{Ha} describes the Hamiltonian approach. These formalisms are developed in an analogous way to the standard contact Lagrangian and Hamiltonian formalisms. We finish this section defining the Legendre transformation on the context of Lie algebroids and we establish the equivalence between both formalisms--Lagrangian and Hamiltonian--when the Lagrangian function is hyperregular. Several examples are also studied in this section.

In section  \ref{legen-lie-subalg-2023}, we introduce the notion of a Legendrian Lie subalgebroid of a contact Lie algebroid. The Hamilton–Jacobi theory \cite{AM-1978} is an alternative formulation of classical mechanics, equivalent to other formulations such as Hamiltonian mechanics, and Lagrangian mechanics for regular systems. For contact systems, it was introduced in \cite{LLM-21, LS}, and it has been extended to many other settings \cite{CGMMMR-2017, CLPRM-2018, ELLMS-2021, V-2012}; this theory is also closely related to the integrability of Hamiltonian systems \cite{LLLR-23, V-2017}. So, we think that it is relevant to extend the theory to the Lie algebroid setting. Therefore, this is done in section \ref{hje}, and reinterpreted in terms of Legendrian subalgebroids.

In this paper, we do not deal explicitly with reduction issues (see for instance, \cite{CeMaRa, L-2001}) which will be discussed in a future paper.

All manifolds and maps are $C^\infty$ and sum over crossed repeated
indices is understood.

\section{Contact geometry}
\label{sec:contact-geometry}

In this section we review the geometric structures necessary to describe the contact formalism of dissipative mechanical systems. For more details, see  \cite{LL-19, Gaset2020} (see also \cite{BCT-2017, dLR}). For general studies of contact geometry in the Riemannian setting, see \cite{Bl1-1976, Bl1-2002, BW-1958}.

\begin{definition}\label{dfn:contact-manifold}
	Consider a smooth manifold $M$ of odd dimension $2n+1$. A differential form $\eta\in\Omega^1(M)$ such that $\eta\wedge(\d\eta)^{ n}$ is a volume form in $M$ is a \textbf{contact form}\index{contact!form}. In this case, $(M,\eta)$ is said to be a \textbf{contact manifold}\index{contact!manifold}.
\end{definition}

\begin{remark}
There is a more general definition of a contact structure on a connected oriented manifold $M$, which it can be seen as an equivalence class of $1$-forms satisfying $\eta \wedge (\d \eta)^n \neq 0$ everywhere on $M$, where two $1$-forms $\eta$, $\eta'$ are equivalent if there exists a nowhere vanishing function $f$ such that $\eta' = f \eta$ (see Section 2.1 in \cite{BOY}).
\end{remark}

Notice that the condition $\eta\wedge(\d\eta)^{ n}\neq 0$ implies that the contact form $\eta$ induces a decomposition of the tangent bundle $T M$ in the form $T M = \ker\eta\oplus\ker \d\eta \equiv D^{\rm C}\oplus D^{\rm \Rb}$.

\begin{proposition}
	Given a contact manifold $(M,\eta)$, there exists a unique vector field $\Rb \in\X(M)$, called \textbf{Reeb vector field}\index{Reeb!vector field}, such that
	\begin{equation}\label{eq:Reeb-condition}
			i_{\Rb}  \d\eta = 0\,, \qquad
			i_{\Rb}  \eta = 1\,.	
	\end{equation}
\end{proposition}
The Reeb vector field $\Rb$ generates the distribution $D^{\rm \Rb}$, called the \textbf{Reeb distribution}\index{Reeb!distribution}.


\begin{theorem}[Darboux theorem for contact manifolds]
	Consider a contact manifold $(M,\eta)$ of dimension $2n+1$. Then, around every point $p\in M$ there exists a local chart $(U,q^i,p_i,s)$, $i = 1,\dotsc,n$, such that
	$$ \restr{\eta}{U} = \d s - p_i\d q^i\,. $$
	These coordinates are called \textbf{Darboux}\index{coordinates!Darboux, for a contact manifold}, \textbf{natural}\index{coordinates!natural, for a contact manifold} or \textbf{canonical coordinates}\index{coordinates!canonical, for a contact manifold} of the contact manifold $(M,\eta)$.
\end{theorem}
Notice that Darboux coordinates are a particular case of adapted coordinates and hence, in Darboux coordinates, the Reeb vector field is
\begin{equation}\label{eq:coordinates-Reeb}
	\restr{ \Rb }{U} = \parder{}{s}\,.
\end{equation}

\begin{example}[Canonical contact structure]\label{ex:canonical-contact-structure}
{\rm 
Let $Q$ be a smooth manifold of dimension $n$. Then, the product manifold $T^* Q\times\R$ has a canonical contact structure given by the 1-form $\eta = \d s - \theta$, where $s$ is the canonical coordinate of $\R$ and $\theta$ is the pull-back of the Liouville 1-form $\theta_\circ\in\Omega^1(T^* Q)$ by the projection $T^* Q\times\R\to T^* Q$. Taking coordinates $(q^i)$ on $Q$ and natural coordinates $(q^i,p_i)$ on $T^\ast Q$, the local expression of the contact 1-form is
\begin{equation}\label{caneta}
\eta = \d s - p_i \d q^i\,.     
\end{equation} 
We also have that $\d\eta = \d q^i\wedge\d p_i$ and hence, the Reeb vector field is $\Rb = \tparder{}{s}$.
}
\end{example}

Given a contact manifold $(M,\eta)$, we have the $\Cinfty(M)$-module isomorphism

$$
	\begin{array}{rccl}
		\flat\colon & \X(M) & \longrightarrow & \Omega^1(M) \\
		& X & \longmapsto & i_X\d\eta + \left(i_X\eta\right)\eta
	\end{array}
$$
\begin{remark}\rm
	Notice that with this isomorphism in mind, we can define the Reeb vector field in an alternative way as $\Rb  = \flat^{-1}(\eta)$.\index{Reeb!vector field}
\end{remark}

\subsection{Contact Hamiltonian systems}
\label{sec:contact-Hamiltonian-systems}

This section reviews the concept of a contact Hamiltonian system and gives two different characterizations of the contact Hamiltonian vector field. 

\begin{definition}
	Given a contact manifold $(M,\eta)$, for every $H\in\Cinfty(M)$, we define its contact Hamiltonian vector field (or just Hamiltonian vector field) as the unique  vector
field $X_H$ satisfying
 \begin{equation}\label{eq:contact-hamilton-equations-fields0}
					\flat (X_H) = \d H - (\Rb(H)+H) \eta.
			\end{equation}
	
\end{definition}

In Darboux coordinates, this is written as follows

\begin{equation*}
	X_H = \parder{H}{p_i}\parder{}{q^i} - \left(\parder{H}{q^i} + p_i\parder{H}{s}\right)\parder{}{p_i} + \left(p_i\parder{H}{p_i} - H\right)\parder{}{s}\,.
\end{equation*}

An integral curve $\gamma(t) = (q^i(t),p_i(t),s(t))$ of this vector field  $X_H$ satisfies the contact Hamilton equations
\begin{equation}\label{eq:contact-hamiltonian-equations-darboux-coordinates}
			\dot q^i = \parder{H}{p_i}\,,\quad
		\dot p_i = -\left(\parder{H}{q^i} + p_i\parder{H}{s}\right)\,,\quad
		\dot s = p_i\parder{H}{p_i} - H\,.
\end{equation}

These equations are a generalization of the conservative Hamilton equations. We recover
this particular case when $\Rb(H)=0$. That is, when $H$ does not depend on $s$.

 The following proposition gives us two equivalent ways of writing equations \eqref{eq:contact-hamilton-equations-fields0}:

\begin{proposition}\label{prop:dissipation-Hamiltonian-mechanics}

Let $H:M\to \r$ be a Hamiltonian function. The following statements are
equivalent:

1. $X_H$ is the Hamiltonian vector field of $H$,

2. $X_H$ satisfies 
\begin{equation}\label{eq:contact-hamilton-equations-fields}
    i_{X_H}\d\eta =\d H-\Rb(H)\eta , \qquad i_{X_H}\eta=-H .
\end{equation}
\end{proposition}

\begin{definition}
 A contact Hamiltonian system is a triple $(M,\eta,H)$, where $(M,\eta)$ is a contact
manifold and $H:M\to \r$  is a smooth real function on $M$ that we will refer to  the Hamiltonian function.
\end{definition}
  The contact Hamiltonian vector fields model the dynamics of dissipative mechanical systems. As
opposed to the case of symplectic Hamiltonian systems, the evolution does not preserve
the energy since\begin{equation}\label{eq:dissipation-Hamiltonian-mechanics}
		X_H(H) = - \Rb(H) H\,,
	\end{equation}
	which expresses the \textsl{dissipation of the Hamiltonian function}.

\subsection{Contact Lagrangian systems}\label{CLS}

Let $Q$ be a manifold with dimension $n$ and coordinates $(q^i)$ and consider the product manifold $T Q \times\r$ with natural coordinates $(q^i,v^i,s)$.
The \textbf{vertical endomorphism}\index{vertical endomorphism}
$ \mathcal{J}\colon T(  T Q \times\r)\to T(  T Q \times\r)\,$ and the \textbf{Liouville vector field}\index{Liouville vector field} $\Delta\in\X(T Q \times\r)$ are the natural extensions of the vertical endomorphism and the Liouville vector field on $T Q$ to $T Q \times \r$ (see \cite{dLR} for more details). The local expressions of these objects in Darboux coordinates are
$$ \mathcal{J} = \parder{}{v^i}\otimes\d q^i\ ,\quad \Delta = v^i\parder{}{v^i}\,. $$

\begin{definition}
	Consider a path $c_s\colon I\subset\r\to Q\times\r$, where $c_s(t)  = (q^i(t),s(t))$. The \textbf{prolongation}\index{prolongation!of a path to $T Q \times\r$} of $c_s$ to $T Q \times\r$ is the path
	$$ c_s'= (\dot c,s)\colon I\subset\r\to T Q \times\r\,,\quad c_s'(t) = (q^i(t), \dot q^i(t), s(t))\,. $$
 The path $c_s'(t)$ is said to be \textbf{holonomic}.
\end{definition}

\begin{definition}
	A vector field $\Gamma\in\X(T Q \times\r)$ is said to satisfy the \textbf{second-order condition}\index{second order condition} or to be a \rm{\textsc{sode}} if all its integral curves are holonomic.
\end{definition}
The following proposition gives an alternative characterization of \textsc{sode}s using the canonical structures defined above:
\begin{proposition}
	A vector field $\Gamma\in\X(T Q \times\r)$ is a \rm{\textsc{sode}} if, and only if, $\mathcal{J}\circ\Gamma = \Delta$.
\end{proposition}
The local expression of a \textsc{sode} is
$ \Gamma = v^i\parder{}{q^i} + f^i\parder{}{v^i} + g\parder{}{s}\,. $
Hence, in coordinates, a \textsc{sode} defines a system of differential equations of the form
$$
		\frac{\d^2q^i}{\d t^2} = f^i(q^i, \dot q^i, s)\,,\quad
		\frac{\d s}{\d t} = g(q^i, \dot q^i, s)\,.	
$$

\begin{definition}
  Let $L\colon T Q  \times\r\to\r$ be a Lagrangian function.
\begin{itemize}

 \item     
$L$  is said to be regular  if its Hessian matrix with respect to the velocities
$$(W_{ij} = \parderr{L}{v^j}{v^i})$$  is non singular.

 \item The associated \textbf{Lagrangian energy}  is  $E_L = \Delta(L) - L\in\Cinfty(T Q \times\r)$.
 
\item The \textbf{Cartan forms}\index{Cartan forms} associated to $L$ are
	\begin{equation*}
		\theta_L =  \d L\circ J \in\Omega^1(T Q \times \r)\ ,\quad \omega_L = -\d\theta_L\in\Omega^2(T Q \times\r)\,.
	\end{equation*}
 
	\item The   $1$-form  
	$$ \eta_L = \d s - \theta_L\in\Omega^1(T Q \times\r)\,, $$
	      is a {\bf contact form} on $T Q \times\r$ if, and only if, $L$ is regular. 
 In this case, the triple $(T Q \times\r,\eta_L,L)$ is called a \textbf{contact Lagrangian system}\index{contact!Lagrangian system}.
 \end{itemize}
\end{definition}

In natural coordinates $(q^i, v^i, s)$ on $T Q \times\r$, the contact Lagrangian form $\eta_L$ is
$$ \eta_L = \d s - \parder{L}{v^i}\d q^i\,, $$
and hence  $\d\eta _L = \omega_L$,   is given by
$$ \d\eta _L = -\parderr{L}{s}{v^i}\d s\wedge\d q^i - \parderr{L}{q^j}{v^i}\d q^j\wedge\d q^i - \parderr{L}{v^j}{v^i}\d v^j\wedge\d q^i\,. $$

Every contact Lagrangian system $(T Q \times\r, \eta_L, L)$ has associated a contact Hamiltonian system $(T Q \times\r, \eta_L, E_L)$. From \eqref{eq:Reeb-condition}, we have that the Reeb vector field $\Rb_L\in\X(T Q \times\r)$ for this contact Hamiltonian system is given by the conditions
\begin{equation*}
		i_{\Rb_L}\d\eta _L = 0\,,\quad
		i_{\Rb_L}\eta_L = 1\,.
\end{equation*}
Its local expression in natural coordinates $(q^i,v^i,s)$ is
$$\Rb_L = \parder{}{s} - W^{ji}\parderr{L}{s}{v^j}\parder{}{v^i}\,, $$
where $(W^{ij})$ is the inverse of the Hessian matrix of the Lagrangian $(W_{ij})$, that is, $W^{ij}W_{jk} = \delta^i_k$.



\begin{definition}
	Consider a contact Lagrangian system $(T Q \times\r, \eta_L, L)$.

	The \textbf{contact Lagrangian equations} for a vector field $X\in\X(T Q \times\r)$ are
	\begin{equation}\label{eq:contact-Lagrangian-equations}
			i_X \d\eta _L = \d E_L - \Rb_L(E_L)\eta_L\,,\quad
			i_X \eta_L = -E_L\,.
				\end{equation}
	
	The vector field $X_L\in\X(T Q \times\r)$ solution to these equations is a {\sc sode}, and it is called the \textbf{contact Lagrangian vector field} (it is a contact Hamiltonian vector field for the function $E_L$).
\end{definition}
Let us observe that if
$\gamma(t) = (q^i(t), v^i(t), s(t))$ is an integral curve of $X_L$, from \eqref{eq:contact-Lagrangian-equations} we obtain
\begin{align}
    v^{i}(t) &= \dot{q}^{i}(t)\,, \\ 
	\parderr{L}{v^j}{v^i}\dot v^j + \parderr{L}{q^j}{v^i} v^j + \parderr{L}{s}{v^i}\dot s - \parder{L}{q^i} = \frac{d}{d t}\left(\parder{L}{v^i}\right) - \parder{L}{q^i} &= \parder{L}{s}\parder{L}{v^i}\,,\label{eq:contact-Euler-Lagrange-1}\\
	\dot s &= L\,,\label{eq:contact-Euler-Lagrange-2}
\end{align}
which coincide with the generalized Euler--Lagrange equations stated in \cite{Her1930}. Observe that $\gamma$ is holonomic, that is, $\gamma (t) = c_{s}'(t) = (q^{i}(t), \dot{q}^{i}(t), s(t))$.

\section{Lie algebroids}\label{algebroids}

In this section, we present some basic facts about Lie algebroids,
including features of the associated differential calculus and results on Lie algebroid morphisms that will be
necessary for the rest of the paper. For further information on groupoids and Lie algebroids, and their roles in differential geometry, see \cite{HM-1990,Mack-1987,Mack-1995}.

\subsection{Generalities on Lie algebroids}

Let $E$ be a vector bundle of rank $m$ over a manifold $Q$ of
dimension $n$, and let $\tau: E\to Q$ be the vector bundle projection.
Denote by $Sec(E)$ the $C^\infty(Q)$-module of sections of
$\tau$. 

A {\it Lie algebroid structure}
$(\lcf\cdot,\cdot\rcf_E,\rho )$ on $E$ is a Lie bracket
$\lcf\cdot,\cdot\rcf_E$ on
the space $Sec(E)$ together with  an {\it anchor map} $\rho : E\to
T Q$ and its, identically denoted, induced $C^\infty(Q)$-module homomorphism
$\rho :Sec(E)\to \vf{(Q)}$, such that the {\it compatibility condition}
\[
\lcf\sigma_1,f\sigma_2\rcf_E=
f\lcf\sigma_1,\sigma_2\rcf_E+(\rho (\sigma_1)f)\sigma_2\,,\]
holds for any smooth functions $f$ on $Q$ and sections $\,\sigma_1,\sigma_2$ of $E$ (here $\rho (\sigma_1)$ is the vector
field on $Q$ given by
$\rho (\sigma_1)( {q})=\rho (\sigma_1( {q}))$). 

The triple
$(E,\lcf\cdot,\cdot\rcf_E,\rho )$ is called a {\it Lie algebroid over $Q$}.
{}From the compatibility condition and the Jacobi identity, it follows
that  $\rho :Sec(E)\to\vf{(Q)}$ is a homomorphism
between the Lie algebras $(Sec(E),\lcf\cdot,\cdot\rcf_E)$ and
$(\vf{(Q)},[\cdot,\cdot])$.





Throughout this paper, the role played by a Lie algebroid is the same as the  tangent bundle of
$Q$. In this way, one regards an element $e$ of $E$ as a generalized
velocity, and the actual velocity ${  v}$ is obtained when we apply the
anchor map to ${  e}$, i.e. ${  v}=\rho ({  e})$.

Let $(q^i)$ be local coordinates on a neighborhood $U$ of $Q$, $i= 1,\ldots,n$, and
$\{e_\alpha\}$ be  a local basis of sections of $\tau$, $\alpha= 1,\ldots, m$.
Given an element ${a_q}\in E$ such that $\tau({a_q})= {q}$, we can write
$a_q=y^\alpha(a_q)e_\alpha( {q})\in E_ {q} ,$ i.e. each
section $\sigma$ is  given locally by $\sigma\big\vert_{U}=y^\alpha
e_\alpha$ and the
coordinates of $a_q$ are $(q^i({q}),y^\alpha({a_q})) \, .$

For {\it the anchor map} $\rho : E\to
T Q$ and its, identically denoted, induced $C^\infty(Q)$-module homomorphism 
$\rho :Sec(E)\to \vf{(Q)}\,,\,\, \left[\rho (\sigma)\right](q)=\rho (\sigma(q))$ , we have 
\begin{equation}\label{roialfa}
 [\rho (e_\alpha)](q)=\rho (e_\alpha(q))=\rho^i_\alpha(q)
\displaystyle\frac{\partial}{\partial q^i} \Big\vert_q  .
\end{equation}

 A Lie algebroid structure on $Q$ is locally determined  as a set of
local {\it structure
functions}  $\rho^i_\alpha,\; \mathcal{C}^\gamma_{\alpha\beta}:Q\to \r$
on $Q$ that are defined by
\begin{equation}\label{structure} \rho (e_\alpha)=\rho^i_\alpha
\displaystyle\frac{\partial}{\partial q^i} ,\quad
\lcf e_\alpha,e_\beta\rcf_E=\mathcal{C}^\gamma_{\alpha\, \beta}e_\gamma\, ,
\end{equation} and
satisfy the  relations
\begin{equation}\label{ecest}\displaystyle\sum_{cyclic(\alpha,\beta,\gamma)}\left(\rho^i_\alpha\displaystyle\frac{\partial
\mathcal{C}^\nu_{\beta\gamma}}{\partial
q^i}+\mathcal{C}^\nu_{\alpha\mu}\mathcal{C}^\mu_{\beta\gamma}\right)=0\;
\, , \quad  \rho^j_\alpha\displaystyle\frac{\partial \rho^i_\beta}{\partial
q^j}- \rho^j_\beta\displaystyle\frac{\partial \rho^i_\alpha}{\partial
q^j}=\rho^i_\gamma \mathcal{C}^\gamma_{\alpha\beta}\;.
\end{equation}
These relations, which are a consequence of the compatibility
condition and Jacobi's identity,  are usually called {\it the structure equations} of the Lie algebroid $E$.

\begin{definition}
A curve
$\widetilde c\colon  I\subseteq\r\to Q$ is called an integral  curve of a
section $\xi$ of $\tau:E\rightarrow Q$ if   $\widetilde c(t)$ is an integral curve of the vector field
$\rho(\xi)$, that is,
\begin{equation}
\rho(\xi)(\widetilde{c}( t))=
\widetilde{c}_*( t)\left(\displaystyle\frac{\d}{\d t}\Big\vert_{
 t}\right). \label{integral:curve}
\end{equation}
\end{definition}
If $\widetilde c$ is written locally as 
$\widetilde c( t)=(q^i( t))$ and $\xi$ as $\xi=y^{\alpha}e_{\alpha}$,
 then we deduce that \eqref{integral:curve} is written in local coordinates as
\begin{equation}\label{algebroid:integral:curve}
  \displaystyle\frac{\d
  q^i}{\d t}\Big\vert_{ t}= \rho^i_\alpha(\widetilde{c}(t))y^\alpha(\widetilde{c}(t))\;.
\end{equation}

A Lie algebroid structure on
$E$ allows us to define {\it the exterior differential of $E$},
$\d^E:Sec(\bigwedge^k E^*)\to Sec(\bigwedge^{k+1} E^*)$,
as follows:
{\small\begin{equation}\label{difE}\begin{array}{lcl}
\d^E\mu \ (\sigma_1,\ldots, \sigma_{k+1})&=&
\displaystyle\sum_{i=1}^{k+1}(-1)^{i+1}\rho (\sigma_i)\mu(\sigma_1,\ldots,
\widehat{\sigma_i},\ldots, \sigma_{k+1})\\\noalign{\medskip} &+&
\displaystyle\sum_{i<j}(-1)^{i+j}\mu(\lcf\sigma_i,\sigma_j \rcf_E,\sigma_1,\ldots,
\widehat{\sigma_i},\ldots, \widehat{\sigma_j},\ldots
\sigma_{k+1})\;,
\end{array}\end{equation}}\noindent for $\mu\in Sec(\bigwedge^k E^*)$ and
$\sigma_1,\ldots,\sigma_{k+1}\in Sec(E)$. It follows that $\d^E$
is a cohomology operator, that is, $(\d^E)^2=0$.

In particular, if $f:Q\to\r$ is a  smooth real function then
$\d^E f\in E^*$ is given by 
$\d^E f(\sigma)=\rho (\sigma)f,$ for $\sigma\in Sec(E)$, so we have that
\begin{equation}\label{defunc}
\d^E f = \rho^i_\alpha \, \ds\frac{\partial f}{\partial q^i} \, e^\alpha ,
\end{equation}
where $\{e^\alpha\}$ is the dual basis of
$\{e_\alpha\}$.
Locally, the
exterior differential is determined by
\begin{equation}\label{difEloc}\d^Eq^i=\rho^i_\alpha e^\alpha\quad \makebox{and}\quad
\d^E e^\gamma=-\displaystyle\frac{1}{2}\mathcal{C}^\gamma_{\alpha\beta}e^\alpha\wedge
e^\beta\,.\end{equation}

Indeed, from  \eqref{structure} we deduce
$\d^Eq^i(e_\alpha)=\rho (e_\alpha)(q^i)=\rho^j_\alpha
\frac{\partial}{\partial q^j} (q^i)=\rho^i_\alpha$. Then $\d^Eq^i=\rho^i_\alpha\, e^\alpha .$ Similarly, $\d^Ee^\gamma(e_\alpha,e_\beta)=\rho (e_\alpha)(e^{\gamma}(e_\beta)) - \rho (e_\beta)(e^{\gamma}(e_\alpha))-e^{\gamma}(\lcf e_\alpha,e_\beta \rcf_{E})=-\mathcal{C}^\gamma_{\alpha\beta}$.

The usual Cartan calculus extends to the case of Lie algebroids: for
every section $\sigma$ of $E$ we  have a derivation $\imath_\sigma$
(contraction) of degree $-1$ and a derivation
$\mathcal{L}_\sigma=\imath_\sigma\circ \d + \d\circ \imath_\sigma$
(the Lie derivative) of degree $0$; for more details, see
\cite{Mack-1987,Mack-1995}.

 Let $(E,\lcf\cdot,\cdot\rcf_{E},\rho )$ and
$(E',\lcf\cdot,\cdot\rcf_{E'}, \rho')$ be two Lie algebroids over $Q$ and
$Q'$ respectively, then a morphism of vector bundles $(F,f)$ of $E$ on $E'$
\[\xymatrix@C=15mm @R=6mm{
E \ar[r]^F \ar[d]_\tau & E' \ar[d]^{\tau'} \\
Q \ar[r]^f & Q' 
}
\]
is said to be a {\it Lie
algebroid morphism} if 
\begin{equation}\label{lie morph}
\d^E ((F,f)^*\sigma')=(F,f)^*(\d^{E'}\sigma ')\,,\quad\makebox{ for all }
\sigma '\in Sec(\textstyle\bigwedge^k (E')^*)\makebox{ and for all
$k$.}
\end{equation}
Here $(F,f)^*\sigma '$ is the section of the vector
bundle $\bigwedge^k E^*\to Q$ defined by
\begin{equation}\label{pullsec}
((F,f)^*\sigma ')_ {q}({a}_1,\ldots,{a}_k) = \sigma_{f( {q})}' (F({ a}_1),\ldots,
F({ a}_k))\,,
\end{equation}
for $ {q}\in Q$ and ${ a}_1,\ldots,
{ a}_k\in E_ {q}$. 
In  particular, if $Q=Q'$ and
$f=id_Q : Q \to Q$ then the pair $(F,f)$ is a Lie algebroid morphism if, and only if,
\[\lcf F\circ\sigma_1,F\circ\sigma_2\rcf_{E'} =
F\lcf\sigma_1,\sigma_2\rcf_{E},\quad
\rho'(F\circ\sigma)=\rho (\sigma),\]
for $\sigma,\sigma_1,\sigma_2\in Sec(E)$.


Finally, we review the notion of a Lie subalgebroid.
\begin{definition}\label{subalg}
Let $(E,\lcf\cdot,\cdot\rcf_{E},\rho)$ and $(F, \lcf\cdot,\cdot\rcf_{F},\rho')$ be two Lie algebroids over the manifolds $Q$ and $N$, respectively. A {\it Lie subalgebroid} is a morphism of Lie algebroids $j: F \to E$ , $i: N \to Q$ such that the pair $(j,i)$ is a monomorphism of vector bundles and $i$ is an injective inmersion (see \cite{HM-1990}).
\end{definition}

\subsection{Examples of Lie algebroids and Lie subalgebroids}

\begin{example} {\bf (Tangent 
bundle)}
{\rm 
The standard example of a Lie algebroid is the
tangent bundle of a manifold $Q$. In this case, the space of sections is just the set of
vector fields on $Q$ and the Lie  bracket of sections      is induced by
the standard Lie bracket of vector fields on $Q$. The anchor map is the identity.

Let $N$ be a submanifold of $Q$, then $TN$ is a Lie subalgebroid of $TQ.$ Now, let $\mathcal{D}$ be a completely integrable distribution on a
manifold $Q$. $\mathcal{D}$ equipped with the bracket of vector
fields is a Lie algebroid over $Q$ since
$\tau_{TQ}\mid_{\mathcal{D}}:\mathcal{D}\to Q$ is a vector bundle. The anchor
map is the inclusion $i_{\mathcal{D}}:\mathcal{D}\to TQ$
($i_{\mathcal{D}}$ is a Lie algebroid monomorphism). Hence, $\mathcal{D}$ is a Lie subalgebroid of the Lie algebroid $\tau_{TQ}:TQ\to Q$. Likewise, if $N$ is an integrable manifold of $\mathcal{D}$, then $\mathcal{D}|_{N}$ is a Lie subalgebroid of $\mathcal{D}$.

}
\end{example}

\begin{example} {\bf (Lie algebra)}
{\rm  Let $\mathfrak{g}$ be a \textit{finite dimensional real Lie algebra} and 
$Q=\{q\}$ be a unique point. The vector bundle
$\tau_{\mathfrak{g}}:\mathfrak{g}\to Q$ is a Lie algebroid. The sections of this
bundle can be identified with the elements of $\mathfrak{g}$, and therefore we can consider as the Lie bracket the structure of the
Lie algebra induced by $\al$, and denoted by $[\cdot,\cdot]_{\mathfrak{g}}$. Since
$TQ=\{0\}$ one may consider the anchor map $\rho\equiv 0$.

Moreover, if $\mathfrak{h}$ is a Lie
subalgebra of $\mathfrak{g}$ and we consider the Lie algebroid induced by
$\mathfrak{g}$ and $\mathfrak{h}$ over a point, then $\mathfrak{h}$
is a Lie subalgebroid of $\mathfrak{g}$.}
\end{example}

\begin{example}
{\bf(Action Lie Algebroid)}
{\rm 
Let $\phi:Q\times G\to Q$ be an action of $G$ on the manifold $Q$, where
$G$  is a Lie group.  {\it The vector bundle $\tau_{Q\times\mathfrak{g}}:Q\times\mathfrak{g}\to
Q$ is a Lie algebroid over $Q$.} The induced anti-homomorphism between the
Lie algebras $\mathfrak{g}$ and $\mathfrak{X}(Q)$ by the action is determined by 
$\Phi:\mathfrak{g}\to\mathfrak{X}(Q)$,
$\xi\mapsto\xi_Q$, where $\xi_{Q}$ is the infinitesimal generator of the
action for $\xi\in\mathfrak{g}$.

The anchor map $\rho:Q\times\mathfrak{g}\to TQ$ is defined by $\rho(q,\xi)=-\xi_{Q}(q)$, and
the Lie bracket of sections is given by the Lie algebra structure on
$Sec(\tau_{Q\times\mathfrak{g}})$ as
$${\br{\hat{\xi}}{\hat{\eta}}}_{Q\times\al}(q)=(q,[\xi,\eta])=\widehat{[\xi,\eta]}(q) ,$$
for $q\in Q$, where $\hat{\xi}(q)=(q,\xi)$, $\hat{\eta}(q)=(q,\eta)$ for
$\xi,\eta\in\al$. The triple $(Q\times\al, {\br{\cdot}{\cdot}}_{Q\times\al},\rho)$ is called \textit{Action Lie algebroid}.

Let $N$ be a submanifold of $Q$ and $\mathfrak{h}$ be a Lie
subalgebra of $\mathfrak{g}$ such that the infinitesimal generators
of the elements of $\mathfrak{h}$ are tangent to $N;$ that is, the
application \begin{eqnarray*}&&\mathfrak{h}\to\mathfrak{X}(N)\\
&&\xi\mapsto \xi_{N} \end{eqnarray*} is well defined. Thus, the
action Lie algebroid $N\times\mathfrak{h}\to N$ is a Lie
subalgebroid of $Q\times\mathfrak{g}\to Q$.
}
\end{example}
\begin{example}\label{Atiyah case}{\bf(Atiyah (gauge) algebroid)}

{\rm 
Let $G$ be a Lie group and assume that $G$ acts freely and properly on $Q$ and we  
denote by $\pi:Q\to Q/G$ the associated principal
bundle. 

The tangent lift of the action gives a free and proper
action of $G$ on $TQ$. Thus, we can consider the fibration $\tau:TQ/G\to Q/G$ given by
$\tau([v_q])=\pi(q)$. It can be
proved that $\tau$  is a vector bundle whose fiber over a point $\pi(q)\in Q/G$ is isomorphic $T_qQ$.

The sections of $\tau: TQ/G \to Q/G$  may be identified  with  the  vector fields on $Q$ which are  
 invariant  by the action $\phi:G\times Q\to Q$, that is,
$$Sec(TQ/G)=\{X\in\mathfrak{X}(Q)\mid
X \hbox{ is $G$-invariant}\}.$$  
Since all $G$-invariant vector fields are $\pi$-projectable
 and the standard Lie bracket on vector fields
is closed with respect to $G$-invariant vector fields, we can define a Lie algebroid structure on $\widehat{TQ}:= TQ/G \to \widehat{Q} := Q/G$, where the anchor map $\rho:\widehat{TQ}\to T(\widehat{Q})$ is given by
$\rho([v_q])=T_{q}\pi(v_q)$


Additionally, let $N$ be a $G$-invariant submanifold of $Q$ and
$\mathcal{D}_{N}$ be a $G-$invariant integrable distribution over $N.$ We may
consider the vector bundle
$\widehat{\mathcal{D}_{N}}=\mathcal{D}_{N}/G\to N/G=\widehat{N}$
and endow it with a Lie algebroid structure. The sections of
$\widehat{\mathcal{D}_{N}}$ are $$Sec(\widehat{\mathcal{D}}_{N})=\{X\in\mathfrak{X}(N)\mid X\hbox{ is $G$-invariant and }X(q)\in\mathcal{D}_{N}(q), \forall q\in N\}.$$
The standard bracket of vector fields on $N$ induces a Lie algebra
structure on $Sec(\widehat{\mathcal{D}}_{N}).$ The anchor map is
the canonical inclusion of $\widehat{\mathcal{D}}_{N}$ on
$T\widehat{N}$ and $\widehat{\mathcal{D}}_{N}$ is a Lie subalgebroid
of $TQ/G\to Q/G.$
}
\end{example}

\subsection{The prolongation of a Lie algebroid over a fibration.}\protect\label{Sec 5.2.}

In this subsection we recall a particular
kind of Lie algebroid that will be used later (see \cite{HM-1990}, for more details).

If $(E,\lcf\cdot,\cdot\rcf_E,\rho)$ is a Lie algebroid of rank $m$ over a smooth manifold $Q$ of dimension $n$, and   $\pi:P\to Q$ is a fibration, then 
$$
\widetilde\tau_P\colon\mathcal{T}^EP= \bigcup_{p\in P}
\mathcal{T}^E_pP \to P,
$$
where
\begin{equation}\label{prol}
\mathcal{T}^E_pP=\{(a_{\pi(p)},v_p)\in E\times T P\, / \,
\rho(a_{\pi(p)})=T\pi(v_p)\}
\end{equation}
  is a Lie algebroid called {\it the prolongation of the Lie algebroid $E$ over $\pi:P\to Q$}, 
 where $T\pi:T P\to T Q$
denotes the tangent map to $\pi$.  The anchor map of this Lie algebroid is
\[\begin{array}{rrcl} \rho^{\pi}\colon
&\mathcal{T}^EP\equiv E\times_{T Q}T P&\to& T P\\
 & (a_{\pi(p)},v_p)&\mapsto &\rho^{\pi}(a_{\pi(p)},v_p)=v_p\,.
\end{array}\] The Lie bracket structure on the space of sections of $\mathcal{T}^EP$ will be given shortly.

In this paper we consider
two particular prolongations, one  over $P = E\times \r\to Q$ and the other over
 $P = E^*\times \r \to Q$.

The following diagram collect the different projections defined from $\mathcal{T}^EP$, that will be used throughout the chapter 
$$\xymatrix@=8mm{ \mathcal{T}^EP\equiv
E\times_{T Q}T P\ar[d]_-{   \tau_1}\ar[r]^-{\rho^\pi}\ar@/^{8mm}/[rr]^-{   \widetilde\tau_P }
& T P\ar[r]^-{\tau_P }\ar[d]^-{T\pi} & P\ar[d]^-{\pi}
\\
E\ar[r]^-{\rho} & T Q\ar[r]^-{\tau_Q} & Q }$$ 
where
\begin{equation}\label{projection prol}
\tau_1(a_{\pi(p)},v_p)=a_{\pi(p)} \,, \qquad
\rho^{\pi}(a_{\pi(p)},v_p)= v_p \,, \qquad 
\widetilde\tau_P (a_{\pi(p)},v_p)= p,
\end{equation}
being $a_{\pi(p)}\in E,\; v_p\in T_pP$ and ${
p}\in P$.

Now we will describe some objects related to
$\mathcal{T}^EP$.
If $(q^i,u^\ell\,)$ are local coordinates on $P$ and $(q^{i},y^{\alpha})$ are local coordinates on $E$ adapted to the local basis of section $\{e_\alpha\}$ of
$\tau:E\to Q$, then 
the induced local coordinate system $(q^i,z^\alpha,u^\ell,\dot{u}^\ell)$ on $\mathcal{T}^EP$, $i=1,\ldots,n$, $\alpha=1,\ldots,m$, $\ell=1,\ldots,n'$, is
\begin{equation}\label{local coord k-prol}\begin{array}{lcllcl}
 q^i(a_{\pi(p)},v_p) &=& q^i(\pi(p))
  \;,\quad &
u^\ell(a_{\pi(p)},v_p)&=&u^\ell(p)\;,
\\\noalign{\medskip}
z^\alpha(a_{\pi(p)},v_p)&=&y^\alpha(a_{\pi(p)})
\;,\quad & \dot{u}^\ell(a_{\pi(p)},v_p)
&=&v{_p}(u^\ell)\;. \\
\end{array}
\end{equation}
where
$$a_{\pi(p)}=y^\alpha
(a_{\pi(p)})e_\alpha(\pi(p)) \,, \qquad v_p=v^i \ds\frac{\partial}{\partial
q^i}\Big\vert_p + \dot{u}^\ell\ds\frac{\partial}{\partial
u^\ell}\Big\vert_p\,,$$
and since   $\rho(a_{\pi(p)})=T_p\pi(v_p)\,$, from \eqref{structure}
 we have
\begin{equation}\label{vicoord}
v^i=y^\alpha(a_{\pi(p)})\rho^i_\alpha(\pi(p))
\;,  
\end{equation}
where $\rho^i_\alpha$ is the local expression of the anchor map $\rho:E\rightarrow TQ$.

 A local basis  of sections of $\widetilde\tau_P\colon\mathcal{T}^EP\to P $ is given by the family 
$\mathcal{X}_\alpha , \mathcal{V}_\ell \colon  P \to \mathcal{T}^EP$,
where
\begin{equation}\label{base k-prol}
\mathcal{X}_\alpha(p) =
\left(e_\alpha(\pi(p)), \rho^i_\alpha(\pi(p))\ds\frac{\partial
}{\partial q^i}\Big\vert_p\right) \,,\qquad \mathcal{V}_\ell(p)
=\left(0_{\pi(p)}, \ds\frac{\partial}{\partial
u^\ell}\Big\vert_p\right) .
\end{equation}

From now we will denote by $Sec(\mathcal{T}^EP)$ the set of sections of the projection $   \widetilde\tau_P :\mathcal{T}^EP\to
 P$. Locally, if a section $Z \in Sec(\mathcal{T}^EP)$ writes as
$Z=Z^\alpha\mathcal{X}_\alpha +V^\ell\mathcal{V}_\ell $,
then the expression of the associated vector field is
$$\rho^{\pi}(Z)= \rho^i_\alpha Z^\alpha\derpar{}{q^i} +
V^\ell\derpar{}{u^\ell}\in \vf(P)\,. $$

Thus, one can observe that the map $\rho^{\pi}$ induces a $\mathcal{C}^\infty(P)$-modules homomorphism
$$\rho^\pi\colon
Sec(\mathcal{T}^EP)\to \vf(P) .$$

The Lie bracket structure on $Sec(\mathcal{T}^EP)$ can be defined from its value on the elements of the local basis  $\{\mathcal{X}_\alpha,\,\mathcal{V}_\ell\}$, which it is characterized by
 the relations
\begin{equation}\label{lie brack k-prol}\begin{array}{lll}
\lcf\mathcal{X}_\alpha,\mathcal{X}_\beta\rcf^{\pi}=
\mathcal{C}^\gamma_{\alpha\beta}\mathcal{X}_\gamma \,, &
\lcf\mathcal{X}_\alpha,\mathcal{V}_\ell\rcf^{\pi}=0 \,, &
\lcf\mathcal{V}_\ell,\mathcal{V}_\varphi\rcf^{\pi}=0\,,
\end{array}
\end{equation}
where $\mathcal{C}^\gamma_{\alpha\beta}$ are the structure functions associated with the Lie bracket of sections of $E\rightarrow Q$.

The exterior differential
$$\d^{\mathcal{T}^EP}\colon  {\it Sec}(\bigwedge^l
(\mathcal{T}^EP)^{\,*})\to {\it Sec}(\bigwedge^{l+1}
(\mathcal{T}^EP)^{\,*})$$ 
is therefore determined by
\begin{equation}\label{dtp}\begin{array}{lclclcl} \d^{\mathcal{T}^EP}q^i &=&
\rho^i_\alpha\mathcal{X}^\alpha\,,&\qquad & \d^{\mathcal{T}^EP}u^\ell ,
&=&
\mathcal{V}^\ell \,,\\
\d^{\mathcal{T}^EP}\mathcal{X}^\gamma &=&
-\displaystyle\frac{1}{2}\mathcal{C}^\gamma_{\alpha\beta}
\mathcal{X}^\alpha\wedge\mathcal{X}^\beta\;,&\qquad &
\d^{\mathcal{T}^EP}\mathcal{V}^\ell&=&0 \,,
\end{array}
\end{equation}
where $\{\mathcal{X}^\alpha,\mathcal{V}^\ell\}$ is the dual basis of
$\{\mathcal{X}_\alpha,\mathcal{V}_\ell\}$.

\begin{example}{\rm

In the case of $E=TQ$, the prolongation $\mathcal{T}^{TQ}TQ$ of this Lie algebroid over the projection $\tau_Q: TQ \to Q$ may be identify with
$TTQ$ with the standard Lie algebroid structure over $TQ$.}
\end{example}

\begin{example}\label{ejemploprolongado1}{\rm  
Let $\al$ be a real Lie algebra of finite dimension. $\al$ is a Lie algebroid over a single point $Q=\{q\}$. We will describe the prolongation  $\mathcal{T}^{\al}\al$ of the Lie algebroid
$\tau_{\al}:\al \to Q=\{q\}$ over the proper fibration $\tau_{\al}:\al \to Q=\{q\}$.

We have the identification
$$\begin{array}{ccl}
     \mathcal{T}^{\al}\al=\{(\xi_1,v_{\xi_2})\in\al\times T\al\}   & \equiv & \al\times(\al\times\al)\\
     (\xi_1,v_{\xi_2})& \equiv & (\xi_1,\xi_2,\xi_3) ,
\end{array}
$$ 
where $v_{\xi_2} \simeq (\xi_2,\xi_3)$.

The vector bundle projection $\tilde\tau_{\al}: \mathcal{T}^{\al}\al\equiv3\al\to\al$ is given by $\tilde\tau_{\al}(\xi_1,\xi_2,\xi_3)=\xi_1$, and the anchor map is $\rho^\tau : \al\times(\al\times\al) \to  T\al$, $\rho^\tau(\xi_1,\xi_2,\xi_3)=(\xi_2,\xi_3)\in T_{\xi_2}\al.$


Let $\{e_{A}\}$ be a basis of the Lie algebra $\al$ and $y^{A}$ the
induced local coordinates on $\al,$ that is,
$\xi=y^{A}e_{A}$. Also this basis   induces a basis of sections   of
$\tilde\tau_{\al}: \mathcal{T}^{\al}\al\equiv3\al\to\al$ as
$$\mathcal{X}_A(\xi)=(\xi,e_{A},0),\qquad \mathcal{V}_A(\xi)=\left(\xi,0,\frac{\partial}{\partial y^{A}}\Big\vert_{\displaystyle\xi}\right),$$ 
see \eqref{base k-prol},
since the anchor map of $\al$ is the zero constant function.

The Lie bracket structure on $Sec(\mathcal{T}^\al \al)$  is characterized by
 the relations \eqref{lie brack k-prol}, that is, 
 
 $$\lcf\mathcal{X}_A,\mathcal{X}_B\rcf^{\tau}= \mathcal{X}_{[e_A,e_B]  } ,  \quad\lcf\mathcal{X}_A,\mathcal{V}_B\rcf^{\tau}=0, \quad 
 \lcf\mathcal{V}_A,\mathcal{V}_B\rcf^{\tau}=0.$$




}\end{example}

\begin{example}\label{ejemploprolongado2}
{\rm Consider a Lie algebra $\al$ acting on a manifold $Q$. Thus,
we have a Lie algebra homomorphism $\al\to\mathfrak{X}(Q)$ mapping
every element $\xi$ of $\al$ to the associated fundamental vector field $\xi_{Q}$ on $Q$.

We consider the  Lie algebroid 
$\tau_{Q\times\al}: Q\times\al \to Q$  with anchor map $$\rho: (q,\xi)\in Q\times \al \longmapsto \rho (q,\xi)=-\xi_Q(q) \in TQ.$$
Identifying $$\,T(Q\times \al)=TQ\times T\al=TQ\times 2\al\, $$ (where $2\al=\al \times \al$), an element of
the prolongation of the Lie algebroid $\tau_{Q\times\al}: Q\times\al \to Q$
$$\mathcal{T}^{Q\times\al}(Q\times\al)
=(Q\times\al)\times_{TQ} T( Q\times\al)=(Q\times\al)\times_{TQ} (TQ\times \al\times \al) $$  over the bundle projection $\tau_{Q\times\al}$ is
of the form $$((q,\xi),(v_{q},\eta,\tilde{\eta})) ,$$ where
$q\in Q$, $v_{q}\in T_{q}Q$ and $(\xi,\eta,\tilde{\eta})\in 3\al$, together with the condition $T\tau_{Q\times\al}(v_{q},\eta, \tilde{\eta}))=\rho(q,\xi)$ which implies that
$v_{q}=-\xi_{Q}(q).$ Therefore, we can identify $\mathcal{T}^{Q\times\al}(Q\times\al)$ with the vector bundle $\tilde\tau_{Q\times\al} : (Q\times\al) \times (\al\times\al) \to Q\times\al$ as follows
$$
\begin{array}{ccc}
     \mathcal{T}^{Q\times\al}(Q\times\al)& \equiv &  (Q\times\al) \times (\al\times\al)  \\
    ((q,\xi),(v_{q},\eta,\tilde{\eta}))   &  \equiv & (q, \xi,\eta, \tilde{\eta}) \, .
\end{array}
$$
Under this identification, the anchor map is given by
$$
\begin{array}{ccccll}
    \rho^{\tau} & : & (Q\times\al) \times (\al\times\al)& \longrightarrow & TQ \times \al \times \al\\ 
 &  &   (q,\xi,\eta,\tilde{\eta}) & \longmapsto & \rho^{\tau}(q,\xi,\eta,\tilde{\eta})= (-\xi_{Q}(q),\eta,\tilde{\eta}) \, .\\
     
\end{array}
$$

Given a basis
$\{e_{A}\}$ of $\al$, the basis $\{\mathcal{X}_A,\mathcal{V}_A\}$ of
sections of $\mathcal{T}^{{Q\times\al}}(Q\times\al) \to Q \times \al$ is given
by
$$\mathcal{X}_A(q,\xi)=(q,\xi,e_{A},0),\quad
\mathcal{V}_A(q,\xi)=(q,\xi,0,e_{A}).$$ 
Finally, the Lie bracket structure on $Sec(\mathcal{T}^{Q\times\al}(Q\times\al))$ is characterized by
 $$\lcf\mathcal{X}_A,\mathcal{X}_B\rcf^{\tau}= \mathcal{X}_{[e_A,e_B]  } ,  \quad\lcf\mathcal{X}_A,\mathcal{V}_B\rcf^{\tau}=0, \quad 
 \lcf\mathcal{V}_A,\mathcal{V}_B\rcf^{\tau}=0.$$

}
\end{example}

\begin{example}\label{ejemploprolongado3} {\rm 
Let us describe the $E$-tangent bundle to $E$ in the case of $E$
being an Atiyah algebroid induced by a trivial principal $G-$bundle
$\pi:G\times Q\to Q.$ In such case, by left trivialization we get
the Atiyah algebroid, the vector bundle $$\tau_{\al\times TQ}:\al\times TQ\to Q.$$ 

For $X\in\mathfrak{X}(Q)$ and $\xi\in\al$,  
we may consider sections $X^{\xi}:Q\to\al\times TQ$ of the
Atiyah algebroid given by $$X^{\xi}(q)=(\xi,X(q))\hbox{ for }q\in Q.$$
Moreover, the anchor map $\rho: \al \times TQ \to TQ$ is defined by $\rho(X^{\xi}(q))=X(q)$ and the Lie bracket of sections is given by $\lcf X^{\xi},Y^{\xi}\rcf_{\al\times 
TQ}=([X,Y]_{TQ},[\xi,\eta]_{\al})$.

Identifying $$\,T(\al \times TQ)=T\al\times TTQ=\al \times \al \times TTQ\, ,$$ 
an element of
the prolongation of the Lie algebroid $\tau_{\al\times TQ}: \al\times TQ \to Q$
$$\mathcal{T}^{\al \times TQ}(\al \times TQ)
=(\al \times TQ)\times_{TQ} T(\al \times TQ)=(\al \times TQ)\times_{TQ} (\al \times \al\times TTQ) $$  over the bundle projection $\tau_{\al\times TQ}$ is of the form 
$$((\xi,v_q),(\eta, \tilde{\eta}, X_{u_q})),$$ 
together with the condition $T\tau_{\al\times TQ}(\eta, \tilde{\eta}, X_{u_q}))=\rho(\xi, v_q)$, which implies that
$u_{q}=v_{q}.$
Thus, we may identify $\mathcal{T}^{\al\times TQ}(\al\times TQ)$ with the vector bundle
$\tilde{\tau}_{\al\times TQ}: \al\times 2\al\times TTQ\to\al\times TQ$ as follows
$$
\begin{array}{ccc}
     \mathcal{T}^{\al\times TQ}(\al\times TQ)& \equiv &  \al\times 2\al\times TTQ  \\
    ((\xi,v_q),(\eta,\tilde{\eta},X_{v_q}))   &  \equiv & (\xi,(\eta,\tilde{\eta}),X_{v_q}) ,
\end{array}
$$
whose vector bundle projection is $\tilde{\tau}_{\al\times TQ}(\xi,((\eta,\tilde{\eta}),X_{v_q}))=(\xi,v_q)$. 

Under this identification, the anchor map is given by
$$
\begin{array}{ccccll}
    \rho^{\tau} & : & \al\times 2\al\times TTQ& \longrightarrow & \al\times\al\times TTQ\\ 
 &  &   (\xi,((\eta,\tilde{\eta}),X)) & \longmapsto & ((\eta,\tilde{\eta}),X) \, .\\
    
\end{array}
$$

If $(\eta,\tilde{\eta})\in 2\al$ and
$X\in\mathfrak{X}(TQ)$, then one may consider the section
$((\eta,\tilde{\eta}),X) : \al \times TQ \to \mathcal{T}^{\al\times TQ}(\al\times TQ)$ given by
$$((\eta,\tilde{\eta}),X)(\xi,v_{q})=(\xi,((\eta,\tilde{\eta}),X(v_q))) , \hbox{
for }(\xi,v_q)\in\al\times T_{q}Q.$$ 
Moreover, the Lie bracket of these sections is given by
$$\lcf((\eta,\tilde{\eta}),X),((\xi,\tilde{\xi}),Y)\rcf^{\tau}=(([\eta,\xi]_{\al},0),[X,Y]_{TQ}).$$ 
}
\end{example}



\section{Contact Lagrangian formalism on Lie algebroids}\label{cftla}

In this section, the contact Lagrangian formalism is extended to the general setting of Lie algebroids.
First,  we will introduce some
  geometric ingredients which are necessary to develop the contact Lagrangian   formalism on Lie algebroids.

\begin{definition}\label{defcontLiealg}
A Lie algebroid $(E,\lcf\cdot,\cdot\rcf_E, \rho)$ of rank $2k+1$ over a manifold $M$ of dimension $n$ is said to be contact if it admits a $1$-section $\eta$ of the vector bundle $\Lambda^1 E^* \to M$ such that
\begin{equation}
\eta \wedge (\d^E \eta)^k \neq 0  \,\, \text{everywhere on M} , 
\end{equation}
where $\d^E : Sec(\bigwedge^l E^*)\to Sec(\bigwedge^{l+1} E^*)$ is the exterior differential of $E$. We say that $\eta$ defines a contact structure on $E$.
\end{definition}
The above definition is equivalent to say that the fibres $(E_x, \eta_x)$ have a contact structure, and therefore they have odd dimension $2 k + 1$. We also notice that $(\d^E \eta)_{\vert_{ker\, \eta}}$ is a non-degenerate $2$-section.

\begin{proposition}
Given a contact Lie algebroid $(E,\eta)$, there exists a unique section $\Rb \in Sec(E)$, called the {\it Reeb section}, such that
\begin{equation}\label{Reebsec-condition}
			i_\Rb \, \d^E\eta = 0\,, \qquad
			i_\Rb \, \eta = 1\,.
\end{equation}
\end{proposition}


The standard  contact  Lagrangian formalism is developed on the  bundle $T Q\times \r$. Since we are thinking of a Lie algebroid $\tau:E\rightarrow Q$ as a substitute of the tangent bundle, it is natural to consider
 the projection map $\pi \colon E\times \r\to Q  
 $ given by $\pi (a_{q},s) =q$. 

If  $(q^i,y^\alpha)$ are local coordinates on
$\tau^{\,-1}(U)\subseteq E$, where $U$ is an open subset of $Q$, adapted to the local basis of section $\{ e_{\alpha} \}$, then the induced local coordinates $( q^i,y^\alpha,s)$ on $\pi^{\,-1}(U)\subseteq E\times \r$ are given by
\begin{equation}\label{cke}
q^i( a_q,s)=q^i(q),\quad
y^\alpha( a_q,s)=y^\alpha( 
a_q),\quad s( a_q,s)=s , 
\end{equation}
where the projection 
$\pi \colon E\times \r\to Q $ is locally given by
  $\pi(q^i,y^\alpha,s)=(q^i)$.


\subsection{The contact Lagrangian prolongation}\label{lagran prolong}\

Let us consider the prolongation  of a Lie algebroid $E$ over the fibration $\pi \colon E\times \r\to Q$. $\tec$ is the vector bundle defined by
\begin{equation}\label{lagrangian prolongation}
\tec=\{(a_q,v_{(b_q,s)})\in E_q\times T_{( b_q,s)}(E\times \r) /\;
\rho (a_q)=T \pi{ ( b_q,s)}(v_{( b_q,s)})\} \,.
\end{equation} 
Therefore by \eqref{vicoord}, we know that
$(a_q,v_{(b_q,s)})\in \tec $
if, and only if,
$$
v_{(b_q,s)}=y^\alpha(a_q) \,\rho^i_\alpha(q)\derpar{}{q^i}+  \dot{y}^\alpha \derpar{}{y^\alpha}       +  \dot{s} \derpar{}{s}  ,   
$$
 being $(q^i,y^\alpha,s, z^\alpha,\dot{y}^\alpha,\dot{s})$, $i=1,\ldots,n$, $\alpha=1,\ldots,m$, the induced local coordinates on $\tec$, where
\begin{equation}\label{local coord cprol h}
\begin{array}{lcllcl}
q^i(a_q,v_{( b_q,s)}) &=& q^i({q})
  \;,\quad & z^\alpha(a_q,v_{( b_q,s)})&=&y^\alpha(b_q)
\;,\\\noalign{\medskip}
y^\alpha(a_q,v_{( b_q,s)})&=&y^\alpha(a_{q})
  \;,\quad &
  \dot{y}^\alpha(a_q,v_{( b_q,s)})&=&v_{( b_q,s)}(y^\alpha)\;,\\\noalign{\medskip}
  s(a_q,v_{( b_q,s)}) &=&s
\;,
\quad & \dot{s}(a_q,v_{( b_q,s)})
&=&v_{( b_q,s)}(s)\;. \\
\end{array}
\end{equation}
 

From Section \ref{Sec 5.2.}, we deduce the following properties of $\tec$.
\begin{enumerate}
\item   The vector bundle $\tec$ with projection $\widetilde\tau_{\er}\colon  \tec\to \er$ given by
$\widetilde\tau_{\er}(a_q,v_{(b_q,s)})=(b_q,s)  $
has a Lie algebroid structure
$(\lcf\cdot,\cdot\rcf^{\pi },\rho^{\pi }\,)$,  
 where the anchor map
$\rho^{\pi }\colon  \tec\to
T( \er)$ given by $\rho^{\pi }(a_q,v_{(b_q,s)})=v_{(b_q,s)}$
is the canonical projection on the second factor. 
We refer to this Lie algebroid as the {\it contact Lagrangian prolongation}.
 
The following diagram shows the different projections defined from $\tec$
$$\xymatrix@=8mm{ \mathcal{T}^E(E\times \r)\equiv
E\times_{T Q}T( E\times \r)\ar[d]_-{   \tau_1}\ar[r]^-{\rho^\pi}\ar@/^{10mm}/[rr]^-{   \widetilde\tau_{E\times \r}}
& T( E\times \r)\ar[r]^-{\tau_{E\times \r}}\ar[d]^-{T\pi} & E\times \r\ar[d]^-{\pi}
\\
E\ar[r]^-{\rho} & T Q\ar[r]^-{\tau_Q} & Q }$$ where
\begin{equation}\label{projection prol}
\hspace{-0.4cm}\begin{array}{lclclclclcl}
   \tau_1(a_q,v_{(b_q,s)})&=&a_q &, &
\rho^{\pi}(a_q,v_{(b_q,s)})&=&v_{(b_q,s)}&, &
    \widetilde{\tau}_{E\times \r}(a_q,v_{(b_q,s)})&=& (b_q,s),
\end{array}\end{equation}
being $a_q \in E,\; v_{(b_q,s)}\in T_{(b_q,s)}(E \times \r)$ and $
(b_q,s) \in E \times \r$.

\item The set
  $ \mathcal{X}_\alpha,\,\mathcal{V}_\alpha, \,  \mathcal{V}_s\,\colon  \er\to\mathcal{T}^E(\er)$ 
given by 
\begin{equation}\label{basec}
\small{ \mathcal{X}_\alpha( b_q,s) =
\left(e_\alpha({q}), \rho^i_\alpha({q})\displaystyle\frac{\partial
}{\partial q^i}\Big\vert_{( b_q,s)}\right), \,\,
\mathcal{V}_\alpha( b_q,s)
=\left(0_{q}, \displaystyle\frac{\partial}{\partial
y^\alpha}\Big\vert_{( b_q,s)}\right) , \,\,
\mathcal{V}_s( b_q,s)
=\left(0_{q}, \displaystyle\frac{\partial}{\partial
s}\Big\vert_{( b_q,s)}\right)\,}
\end{equation}
is a local basis of $Sec(\tec)$, the set of sections of 
$\widetilde\tau_{E\times \r}$ (see \eqref{base k-prol}).

\item   The anchor map $\rho^{\pi }\colon \mathcal{T}^E(\er )\to \mathfrak{X}(\er )$
allows us to associate a vector field on $ \er \stackrel{\xi}{\tol}  \mathcal{T}^E(\er )\stackrel{\rho^{\pi }}{\tol} \mathfrak{X}(\er )  $ with each section $\xi\colon \er \to\tec $.
 
If  locally 
$\xi= \xi_2^\alpha\mathcal{X}_\alpha+\xi_1^\alpha \mathcal{V}_\alpha  +\xi_0\mathcal{V}_s$,
 then
\begin{equation}\label{rholksimc}
\rho^{\pi }(\xi)= \rho^i_\alpha
\xi_2^\alpha\derpar{}{q^i} + \xi_1^\alpha\derpar{}{y^\alpha} +\xi_0\derpar{}{s}\in
\vf(\er)\,.\end{equation}         

\item The Lie bracket of two sections of $\widetilde\tau_{E\times \r}$ is characterized by (see \eqref{lie brack k-prol}):
\begin{equation}\label{lie brack tecL}
\begin{array}{lclclclclcl}
\lcf\mathcal{X}_\alpha,\mathcal{X}_\beta\rcf^{\pi }&=&
\mathcal{C}^\gamma_{\alpha\beta}\mathcal{X}_\gamma\;,
&\quad &
\lcf\mathcal{X}_\alpha,\mathcal{V}_\beta\rcf^{\pi }&=&0
\;, &\quad
&
\lcf\mathcal{X}_\alpha,\mathcal{V}_s\rcf^{\pi }&=&
0\;,
\\\noalign{\medskip}
\lcf\mathcal{V}_\alpha,\mathcal{V}_\beta\rcf^{\pi }&=&0
\;, &\quad &
\lcf\mathcal{V}_\alpha,\mathcal{V}_s\rcf^{\pi }&=& 0\;,
&\quad &
\lcf\mathcal{V}_s,\mathcal{V}_s\rcf^{\pi }&=& 0\,.
\end{array}
\end{equation}

\item If $\{\mathcal{X}^\alpha,
\mathcal{V}^\alpha,\mathcal{V}^s\}$ is the dual basis of  $\{\mathcal{X}_\alpha, \mathcal{V}_\alpha,\mathcal{V}_s\}$, then the exterior differential is given locally, see \eqref{dtp}, by
\begin{equation}\label{difksimc}
\begin{array}{l}
\d^{\tec}q^i =
\rho^i_\alpha\mathcal{X}^\alpha\,,\qquad  \d^{\tec}y^\alpha=\mathcal{V}^\alpha \,,\qquad \d^{\tec}s= \mathcal{V}^s
\\\noalign{\bigskip} 
\d^{\tec}f=\rho^i_\alpha\derpar{f}{q^i}\mathcal{X}^\alpha
+ \derpar{f}{y^\alpha}\mathcal{V}^\alpha + \derpar{f}{s}\mathcal{V}^s\,,\quad \makebox{for all } f\in \mathcal{C}^\infty(\er)
\\\noalign{\bigskip} 
\d^{\tec}\mathcal{X}^\gamma =
-\displaystyle\frac{1}{2}\mathcal{C}^\gamma_{\alpha\beta}\mathcal{X}^\alpha\wedge\mathcal{X}^\beta\, \quad ,\quad
\d^{\tec}\mathcal{V}^\alpha=0, \quad \d^{\tec}\mathcal{V}^s=0\,.\end{array}
\end{equation}
From now on we are going to set the notation $\d = \d^{\tec}$.
\end{enumerate}

\begin{remark}\label{remark equivch}
Note that in the particular case $E= T Q$, the Lie algebroid $\tec$ reduces to $T( T Q\times \r)$.
\end{remark}

\subsection{Liouville sections and  vertical endomorphisms}\

One can define on $\tec$ two families of canonical objects: {\it Liouville section} and {\it vertical endomorphism}; which correspond to the {\it Liouville vector field} and {\it canonical tensor field} on $ T Q\times \r $ of Section \ref{CLS}.

\paragraph{\bf The vertical lifting}

We consider the projection on the first factor
${\tau}_1: \mathcal{T}^E(\er ) \to E$, $\tau_1(a_q,v_{( b_q,s)}) = a_q$.
An element $(a_q,v_{( b_q,s)})$
of $\mathcal{T}^E(\er )$
is said to be vertical if
${\tau}_1(a_q,v_{( b_q,s)})=0_{q}\in
E $. Thus, the vertical elements are   of the form
$(0_{q},v_{( b_q,s)})$ .

In particular, the tangent vector $v_{( b_q,s)}\in T_{( b_q,s)}(\er)$ is $\pi$-vertical, since by  (\ref{lagrangian prolongation}) we have 
  $
\rho (a_q)=T\pi (v_{( b_q,s)}) )\in T_qQ$ 
and   $a_ q=0_q$.

In a local coordinate system  $(q^i,y^\alpha,s)$  on $\er $, if $(a_q,v_{( b_q,s)})\in\tec$ is vertical, then $a_q=0_q$ and
$$v_{( b_q,s)}=  \dot y^\alpha \displaystyle\frac{\partial }{\partial
y^\alpha}\Big\vert_{( b_q,s)}+\dot s\derpar{}{s}\Big\vert_{( b_q,s)}\in
T_{( b_q,s)}( \er )\,.$$

\begin{definition}\label{lvastkec}

The vertical  lifting is defined as the mapping
 \begin{equation}\label{a-levantamientoc}
\begin{array}{rcl}
\Upsilon^{{\mathbf V}}:E\times_Q(\er ) & \longrightarrow &
\mathcal{T}^E(\er )
\\\noalign{\medskip}
 (a_q, (b_{q},s)) & \longmapsto &
\Upsilon^{{\mathbf V}}(a_q,(b_q,s))=\left(0_{q},(a_q)^{V}_{( b_q,s)}\right) \\
\end{array}
\end{equation} 
where $\;(a_q)^{V}_{( b_q,s)}\in
T_{( b_q,s)} (\er )$ is given by
\begin{equation}\label{verticalc}
(a_q)^{V}_{( b_q,s)}f=\displaystyle\frac{d}{dt}\Big\vert_{t=0}f( 
  b_{q}+t a_q )\,,\quad  \end{equation} for an arbitrary function $f\in \mathcal{C}^\infty(\er )$.
\end{definition}
The local expression of
$(a_q)^{V}_{(b_q,s)}$ is
\begin{equation}\label{localvertc}
(a_q)^{V}_{(b_q,s)}=y^\alpha(a_q)\displaystyle\frac{\partial}{\partial
y^\alpha}\Big\vert_{( b_q,s)}\in
T_{( b_q,s)}( \er )\,,
\end{equation}
and therefore $(a_q)^{V}_{( b_q,s)}\in T_{( b_q,s)}(\er )$ is $\pi $-vertical, since
$\pi(q^i,y^\alpha,s)=(q^i)$.

%
%
%
%
%

\subsection{ The vertical endomorphism}\label{endvertA}
The {\rm vertical endomorphism} $S$ on  $\mathcal{T}^E(\er )$ is the map defined by
\begin{equation}\label{jtildeA}
\begin{array}{rccl}
S:&\mathcal{T}^E(\er ) & \longrightarrow &
\mathcal{T}^E(\er )\\\noalign{\medskip}
&(a_q,v_{( b_q,s)})&\longmapsto
&S(a_q,v_{( b_q,s)})=
\Upsilon^{{\mathbf V}}
(a_q,(b_q,s)) \,,
\end{array}
\end{equation} 
which locally writes
$$S(a_q,v_{( b_q,s)})=\left(0_q, y^\alpha(a_q)\displaystyle\frac{\partial}{\partial
	y^\alpha}\Big\vert_{( b_q,s)}\right)=
y^\alpha(a_q)
\mathcal{V}_\alpha( b_q,s) .$$
Now, from \eqref{basec} we have  
$$
S(\mathcal{X}_\alpha( b_q,s) )=\mathcal{V}_\alpha( b_q,s) , \qquad
S(\mathcal{V}_\alpha( b_q,s))=0 , \qquad S(\mathcal{V}_s( b_q,s))=0 ,
$$
and then $S$ has the local expression
\begin{equation}\label{localtildeSAksim}
S=\mathcal{V}_\alpha\otimes\mathcal{X}^\alpha\;.
\end{equation}

\begin{remark}\label{vertical tensor}\
 The endomorphism $S$ defined above will allow us to introduce the concept of {\it Lagrangian section} when we develop the contact Lagrangian formalism on Lie algebroids. Moreover, this mapping will give a characterization of certain sections of $\tec$ which we consider later.
 \end{remark}

 \paragraph{\bf The Liouville section}\label{Lagform1}\
The  {\rm  Liouville section} $\Delta $ is the section of
$\widetilde\tau_{\er }:\mathcal{T}^E(\er )\to \er $ given by
$\Delta ( b_q,s)=
\Upsilon^{\mathbf V}( b_q, (b_q,s))$.
Locally
$$\Delta ( b_q,s) 
=\left(0_q,y^\alpha(b_q)\displaystyle\frac{\partial}{\partial
y^\alpha}\Big\vert_{( b_q,s)}\right)= y^\alpha(b_q)\left(0_q,\displaystyle\frac{\partial}{\partial
y^\alpha}\Big\vert_{( b_q,s)}\right)=  y^\alpha(b_q)\mathcal{V}_\alpha( b_q,s) ,
$$
and thus  $\Delta $ has the local expression
\begin{equation}\label{Liouville kcosim}
\Delta =
y^\alpha\mathcal{V}_\alpha\;.
\end{equation}

  In the standard contact Lagrangian formalism, the Liouville vector field $\Delta$ allows us to define the energy function. Analogously as we will see below, the energy function can be defined in the Lie algebroid setting using the Liouville section $\Delta$.

\subsection{Second order   differential equations ({\sc sode}'s).}
\label{Sec 8.1.2.}

%
%
%
%
%
%


As we saw in Section \ref{CLS}, in the standard  contact Lagrangian formalism one obtains the solutions of the Herglotz equations  as integral curves of certain second-order differential equation ({\sc sode}) on $T Q\times \r$. Now we introduce the analogous object on Lie algebroids.

\begin{definition} A {\rm second order  differential equation ({\sc sode})} $\Gamma$ is a section of $\widetilde\tau_{\er }$ which satisfies the equation
$S(\Gamma)=\Delta$.   
\end{definition}


The local expression of a {\sc  sode} is
$ \Gamma=  
y^\alpha\mathcal{X}_\alpha+  f^\alpha\mathcal{V}_\alpha+  g\mathcal{V}_s \;,$
where $  f^\alpha, g \in\mathcal{C}^\infty(\er )$,
and the associated vector field  $\rho^{\pi }(\Gamma)\in\vf(\er )$ is given by
\begin{equation}\label{ sode asso co}  \rho^{\pi }(\Gamma)=
 \rho^i_\alpha y^\alpha \derpar{}{q^i} +
  f^\alpha\derpar{}{y^\alpha}+g\derpar{}{s} \,.
\end{equation}

Suppose that the curve
$\widetilde c\colon  I\subseteq\r\to \er $ is an integral  curve of a
{\sc  sode} $\Gamma $ (that is, it satisfies Equation \eqref{integral:curve}). If $\widetilde c$ is written locally as 
$\widetilde c( t)=(q^i( t),y^\alpha( t), s(t))$,
 then from  \eqref{ sode asso co} we deduce that \eqref{integral:curve} is locally equivalent to the identities
\begin{equation}\label{integral sect}
  \displaystyle\frac{\d
  q^i}{\d t}\Big\vert_{ t}= \rho^i_\alpha(q^i(t))y^\alpha(t)\;,\quad \displaystyle\frac{ \d y^\alpha}{\d t}\Big\vert_{ t}=f^\alpha(\widetilde c( t))\;, \quad \displaystyle\frac{\d s}{\d t}\Big\vert_{ t}=g(\widetilde c (t)) .
\end{equation}

%
%
%
%

\subsection{ Lagrangian formalism}\

In the remainder of this section, we will develop an intrinsic geometric framework, which allows  us to write the Herglotz equations associated with a Lagrangian function $L\colon \er \to\r$ on a Lie algebroid. We first introduce some geometric elements associated with $L$.

Let us consider
$$
(\mathcal{T}^E(\er ) )_{( b_q,s)} =\{(a_q,v_{( b_q,s)})\in \tec \,/\, a_q\in E, \rho(a_q)=T\pi(b_q,s) (v_{(b_q,s)})\} $$
the fibre of $\mathcal{T}^E(\er )  \to \er$ over the point $( b_q,s)$.

\paragraph{\bf Poincar\'{e}-Cartan and contact sections}\

The {\it Poincar\'{e}-Cartan $1$-section}  $\Theta_L: \er   \to  (\mathcal{T}^E(\er ))^{*}$, 
where $$\Theta_L( b_q,s)\colon (\mathcal{T}^E(\er ) )_{( b_q,s)} \to \r$$  
is the linear map defined by
\begin{equation}\label{ThetaLAc}
\Theta_L( b_q,s)(a_q,v_{( b_q,s)})=\d
L( b_q,s)(S_{( b_q,s)}(a_q,v_{( b_q,s)})) = [\rho^{\pi }(S_{( b_q,s)}(a_q,v_{( b_q,s)}))]L ,
\end{equation}
since the last identity follows from \eqref{rholksimc}, \eqref{difksimc}  and \eqref{localtildeSAksim}.

One can define the following $1$-form $\eta_L$ associated with $L$ as follows
\begin{equation}\label{etal}
\eta_L= \mathcal{V}^s-\Theta_L ,
\end{equation}
then, its differential
$\d\eta_L:\er  \to
\Lambda^2(\mathcal{T}^E(\er ))^{*}, $
is given by
$$
\d\eta_L =\d ( \mathcal{V}^s-\Theta_L)
=\d  \mathcal{V}^s- \d  \Theta_L=- \d  \Theta_L \, .
$$
 
 From
\eqref{rholksimc}, \eqref{localtildeSAksim} and \eqref{ThetaLAc}, we deduce the local expressions of $\Theta_L$ and $\eta_L$
\begin{equation}\label{local thetac}
\Theta_L=\displaystyle\frac{\partial L}{\partial y^\alpha
}\mathcal{X}^\alpha \;,\qquad  \eta_L=\mathcal{V}^s-\displaystyle\frac{\partial L}{\partial y^\alpha
}\mathcal{X}^\alpha  ,
\end{equation}
and from the local expressions \eqref{rholksimc}, \eqref{lie brack
tecL}, \eqref{difksimc}  and \eqref{local thetac}, we obtain
\begin{equation}\label{local omega co}
	\d\eta_L = \left(\rho^i_\beta \displaystyle\frac{\partial
		^2 L}{\partial q^i\partial y^\alpha} + \displaystyle\frac{1}{2}
	\mathcal{C}^\gamma_{\alpha\beta}\displaystyle\frac{\partial L}{\partial
		y^\gamma} \right) \mathcal{X}^\alpha \wedge \mathcal{X}^\beta
	 +
	\displaystyle\frac{\partial ^2 L}{\partial y^\beta\partial y^\alpha}\,
	\mathcal{X}^\alpha \wedge \mathcal{V}^\beta
	+ 	\derpars{L}{s}{y^\alpha}\mathcal{X}^\alpha\wedge \mathcal{V}^s \,.
	\end{equation}

\begin{definition}
We say that the Lagrangian function $L$ is {\it regular} if  the matrix  $\left(\ds\frac{\partial^2L}{\partial y^\alpha\partial y^\beta}\right)$ is non-singular.
\end{definition}
 
\begin{remark}
If the Lagrangian function $L$ is regular, then from \eqref{local thetac} and \eqref{local omega co} we deduce that $\eta_L$ defines a contact structure in the sense of Definition \ref{defcontLiealg}, since
$$\eta_{L}\wedge (\d \eta_{L})^{m} = \det \left(\frac{\partial^2 L}{\partial y^\alpha \partial y^\beta}\right) \mathcal{X}^1 \wedge \dots \wedge \mathcal{X}^m \wedge \mathcal{V}^1 \wedge \dots \wedge \mathcal{V}^m \wedge \mathcal{V}^s ,$$
and the Reeb section $\Rb_L$ associated to $L$, characterized by the two conditions \eqref{Reebsec-condition}, is locally given by
\begin{equation}\label{ReebLloc}
\Rb_L = \mathcal{V}_s - \ds\frac{\partial^2 L}{\partial s \partial y^\beta} \, \left(\ds\frac{\partial^2 L}{\partial y^\alpha \partial y^\beta}\right)^{-1} \mathcal{V}_\alpha .
\end{equation}    
\end{remark}

\paragraph{\bf Energy function}\

The {\it energy function}
$E_L:\er  \to \r$
 associated to the Lagrangian $L$ is
\begin{equation}\label{energy}
E_L=\displaystyle\rho^{\pi }(\Delta )L-L \,.    
\end{equation}
From \eqref{rholksimc} and \eqref{Liouville kcosim} one deduces its local expression
\begin{equation}\label{local enerco}
E_L=\displaystyle y^\alpha\displaystyle\frac{\partial L}{\partial
y^\alpha}- L\;.
\end{equation}

\subsection{Herglotz equations} \

\begin{theorem}\label{algeform co}
Given a regular Lagrangian $L\colon \er \to \r$, since $\eta_L$ is a contact section, there exists a unique section 
$\Gamma_L : \er \to \tec$ of $\widetilde\tau_{E\times \r}$, called the Lagrangian section, satisfying
\begin{equation}\label{ec ge EL co}
 \imath_{\Gamma_L}\eta_L=-E_L\,,\qquad
    \imath_{\Gamma_L}\d\eta _L =\d E_L+
\rho^{\pi}(\Rb_L)(E_L) \, \eta_L \,.
\end{equation}
Moreover,
\begin{enumerate}
    \item $\Gamma_L$ is a {\sc sode}.
    \item If $\widetilde{c} :
 I \subset \r
\to \er \,,\;\widetilde{c}(t)=(  q^i(t),
  y^\alpha (t),  s(t))$ is an integral curve of
$\Gamma_L$, then $\widetilde{c}$ is a solution of the following system of differential equations
\begin{equation}\label{eq E-L contc}
\begin{array}{rcl}   \ds\frac{\d}{ \d
	t     }\left( \ds\frac{\partial  L}{\partial
	y^\alpha }\Big\vert_{\widetilde{c}(t)}\right) &=&
\rho^i_\alpha(q^i(t)) \, \ds\frac{\partial L}{\partial
	q^i}\Big\vert_{\widetilde{c}(t)} -
y^\beta (t) \, \mathcal{C}^\gamma_{\alpha\beta} \, \ds\frac{\partial
	L}{\partial y^\gamma }\Big\vert_{\widetilde{c}(t)} + \displaystyle\frac{\partial L}{\partial y^\alpha
}\Big\vert_{\widetilde{c}(t)}\displaystyle\frac{\partial L}{\partial s
}\Big\vert_{\widetilde{c}(t)} ,
\\\noalign{\medskip}
\ds\frac{\d q^i}{\d t     }\Big\vert_{t} &=&
y^\alpha (t) \, \rho^i_\alpha\;,\qquad \displaystyle\frac{\d s}{\d t}\Big\vert_{ t}=L(\widetilde{c}(t)) ,
\end{array}
\end{equation}
for $i=1,\ldots,n$ and $\alpha=1,\ldots,m$, where $\rho^i_\alpha$ and $\mathcal{C}^\gamma_{\alpha\beta}$ are the structure functions of the Lie algebroid $E$ with respect to the coordinates $(q^i)$ and the local basis $\{e_\alpha\}$.

These equations are the {\bf Herglotz equations} on Lie algebroids.
\end{enumerate}
\end{theorem}

\proof 
As
 $\Gamma_L \in Sec(\tec)$ can be locally written as
\begin{equation}\label{GammaLproof}
\Gamma_L =A^\alpha \mathcal{X}_\alpha+ B^\alpha \mathcal{V}_\alpha+ C  \mathcal{V}_s ,   
\end{equation} 
for some functions $A^\alpha ,B^\alpha,C   \in
\mathcal{C}^\infty(\er )$. 
Now, from \eqref{localtildeSAksim} and \eqref{GammaLproof} we obtain
$$\eta_L(\Gamma_L)= \mathcal{V}^s(\Gamma_L)
-\displaystyle\frac{\partial L}{\partial y^\alpha
}\mathcal{X}^\alpha(\Gamma_L)=C -\displaystyle\frac{\partial L}{\partial y^\alpha
} A^\alpha=-E_L=L-\displaystyle y^\alpha\displaystyle\frac{\partial L}{\partial
y^\alpha} .
$$

On the other hand, from the local expression
\eqref{local omega co}, a straightforward computation in local coordinates shows that
$$\begin{array}{ll}
  \imath_{\Gamma_L}\d\eta_L      = & A^\alpha \,
\derpars{L}{s    }{y^\alpha }\,\mathcal{V}^s   +
A^\alpha \derpars{L}{y^\beta   }{y^\alpha } \mathcal{V}^\beta   
\\\noalign{\medskip} & -
\left(C\derpars{L}{s }{y^\alpha } +
A^\beta \left(\rho^i_\beta\derpars{L}{q^i}{y^\alpha } -
\rho^i_\alpha\derpars{L}{q^i}{y^\beta }+\mathcal{C}^\gamma_{\alpha\beta}
\derpar{L}{y^\gamma }\right)+B^\beta   \derpars{L}{y^\beta   }{y^\alpha }\right)\mathcal{X}^\alpha\;,
\end{array}$$ 
and from the local expressions \eqref{difksimc} and \eqref{local enerco}, we obtain
$$\d E_L + \derpar{L}{s     }\eta_L  =
\left(y^\alpha \derpars{L}{s    }{y^\alpha } +\displaystyle\frac{\partial L}{\partial s
} \right)\mathcal{V}^s    +
y^\alpha \derpars{L}{y^\alpha }{y^\beta   } \mathcal{V}^\beta    +
 \left[ \rho^i_\alpha\left(y^\beta \derpars{L}{q^i}{y^\beta }
- \displaystyle\frac{\partial L}{\partial q^i
}  \right)  -\displaystyle\frac{\partial L}{\partial y^\alpha
}\displaystyle\frac{\partial L}{\partial s
}\right]\mathcal{X}^\alpha  
\,.
$$
Whence it follows that $\Gamma_L:\er \to\tec$ is a solution of the system \eqref{ec ge EL co} if, and only if,
\begin{equation}\label{imp ec el}\begin{array}{c}\begin{array}{rcl}
 A^\alpha \, \derpars{L}{s    }{y^\alpha } &=&
y^\alpha \derpars{L}{s   }{y^\alpha }+\derpar{L}{s }\;,
\\\noalign{\medskip}
A^\alpha \derpars{L}{y^\beta   }{y^\alpha }& =&
y^\alpha \derpars{L}{y^\alpha }{y^\beta   }\;,\end{array}\\\noalign{\medskip}
\begin{array}{ll} & -\left[C\derpars{L}{s    }{y^\alpha } +
A^\beta \left(\rho^i_\beta\derpars{L}{q^i}{y^\alpha }-
\rho^i_\alpha\derpars{L}{q^i}{y^\beta } +
\mathcal{C}^\gamma_{\alpha\beta}\derpar{L}{y^\gamma }\right)+
B^\beta   \derpars{L}{y^\beta   }{y^\alpha } \right]
\\\noalign{\medskip}=&\rho^i_\alpha\left(y^\beta \derpars{L}{q^i}{y^\beta }
- \displaystyle\frac{\partial L}{\partial q^i
}  \right)  -\displaystyle\frac{\partial L}{\partial y^\alpha
}\displaystyle\frac{\partial L}{\partial s
} .
\end{array}\end{array}\end{equation}

Since $L$ is regular, from the second identity of \eqref{imp ec el}, we obtain
$$
A^\alpha =y^\alpha \;,\quad \alpha=1,\ldots,m\,.$$
Therefore $\Gamma_L$ is a {\sc sode}, and from
$$ C -\displaystyle\frac{\partial L}{\partial y^\alpha
} A^\alpha= L-\displaystyle y^\alpha\displaystyle\frac{\partial L}{\partial
y^\alpha} 
$$
we conclude $C=L$. Now, from the last identity on (\ref{imp ec el}) we obtain
 \begin{equation}\label{imp ec el1} 
 L \, \derpars{L}{s    }{y^\alpha } +
y^\beta \left(\rho^i_\beta \, \derpars{L}{q^i}{y^\alpha }
+
\mathcal{C}^\gamma_{\alpha\beta} \, \derpar{L}{y^\gamma }\right)+
B^\beta  \, \derpars{L}{y^\beta   }{y^\alpha } 
= \rho^i_\alpha \, \displaystyle\frac{\partial L}{\partial q^i
}    + \displaystyle\frac{\partial L}{\partial y^\alpha
}\displaystyle\frac{\partial L}{\partial s
} .
\end{equation}

In summary, if a section $\Gamma_L$ is a solution of (\ref{ec ge EL co}), then
 $\Gamma_L$ is a {\sc sode} in $\tec$ and it can be  written locally as follows:
$\Gamma_L = y ^\alpha\,
\mathcal{X}_\alpha +  B^\alpha   \, \mathcal{V}_\alpha  +L\, \mathcal{V}_s \,,$ 
for some functions $B^\alpha   \in \mathcal{C}^\infty(\er )$
satisfying
\begin{equation}\label{eqla3}y^\beta \, \rho^i_\beta \, \derpars{L}{q^i}{y^\alpha }+
B^\beta   \, \derpars{L}{y^\beta   }{y^\alpha } + L \, \derpars{L}{s     }{y^\alpha } =
\rho^i_\alpha \, \derpar{L}{q^i} - y^\beta  \,
\mathcal{C}^\gamma_{\alpha\beta} \, \derpar{L}{y^\gamma }
  + \displaystyle\frac{\partial L}{\partial y^\alpha
}\displaystyle\frac{\partial L}{\partial s
}
\,.\end{equation}

Now, let $\widetilde{c} : I \subset \r \to E \times \r$, $\widetilde{c}( t)=(q^i(t),y^\alpha(t),s(t))$ be an integral curve of the {\sc sode} $\Gamma_L$, that is, an integral curve of the vector field $\rho^\pi(\Gamma_L)$, say
$$\rho^\pi(\Gamma_L)(\widetilde{c}(t)) = \widetilde{c}_* (t) \left(\ds\frac{d}{d t}\Big\vert_{t}\right).$$
From  \eqref{ sode asso co} we deduce that \eqref{int sect co} is locally equivalent to the identities
\begin{equation}\label{integral sect2}
  \displaystyle\frac{\d
q^i}{\d t}\Big\vert_{ t}=
 \rho^i_\alpha(q^i(t))y^\alpha(t)\;,\quad \displaystyle\frac{ \d y^\alpha}{\d t}\Big\vert_{ t}=B^\alpha(\widetilde{c}( t))\;. \quad \displaystyle\frac{\d s}{\d t}\Big\vert_{ t}=L( \widetilde{c}(t)) \;.
\end{equation} 
If we restrict equations \eqref{eqla3} to the image of $\widetilde{c}(t)$ and consider the above identities \eqref{integral sect2}, we obtain
$$
\ds\frac{\d q^i}{\d t}\,\derpars{L}{q^i}{y^\alpha } +
 \displaystyle\frac{ \d y^\beta}{\d t}\,   \derpars{L}{y^\beta}{y^\alpha }  + \ds\frac{\d s}{\d t}\,  
 \derpars{L}{s     }{y^\alpha } =
\rho^i_\alpha \,\derpar{L}{q^i} - y^\beta\,
\mathcal{C}^\gamma_{\alpha\beta}\derpar{L}{y^\gamma }
  +\displaystyle\frac{\partial L}{\partial y^\alpha} \,\displaystyle\frac{\partial L}{\partial s}  ,
$$
or equivalently
$$\begin{array}{rcl}   \ds\frac{\d}{ \d
	t     }\left( \ds\frac{\partial  L}{\partial
	y^\alpha }\Big\vert_{\widetilde{c}(t)}\right) &=&
\rho^i_\alpha \, \ds\frac{\partial L}{\partial
	q^i}\Big\vert_{\widetilde{c}(t)} -
y^\beta (t) \, \mathcal{C}^\gamma_{\alpha\beta} \, \ds\frac{\partial
	L}{\partial y^\gamma }\Big\vert_{\widetilde{c}(t)} + \displaystyle\frac{\partial L}{\partial y^\alpha
}\Big\vert_{\widetilde{c}(t)}\displaystyle\frac{\partial L}{\partial s
}\Big\vert_{\widetilde{c}(t)} ,
\\\noalign{\medskip}
\ds\frac{\d q^i}{\d t     }\Big\vert_{t} &=&
y^\alpha (t) \, \rho^i_\alpha\;,\quad \displaystyle\frac{\d s}{\d t}\Big\vert_{ t}=L(\widetilde{c}(t)) ,
\end{array}
$$  
which are the Herglotz equations on Lie algebroids.
\qed

\begin{remark}
If $E$ is the standard Lie algebroid $T Q$, then $\Theta_L$ and $\eta_L$ are the usual Poincaré--Cartan $1$-form and the contact $1$-form  respectively, associated with the Lagrangian function $L\colon  T Q  \times \r \to \r$ considered in Section \ref{CLS}. The equations of motion are the Herglotz equations given in Section \ref{CLS}.
\end{remark}

\begin{example}{\rm 
    If $E = \mathfrak{g}$ is the Lie algebra of a Lie group $G$ projecting over $Q=\{0\}$, let us consider coordinates $(y^{A})$ on $\al$ associated with the Lie algebra basis $\{e_{A}\}$. Then we obtain Euler-Poincar\'e-Herglotz equations (see \cite{ACLMMP-2023}) \begin{equation*}\
		\begin{split}
			& \frac{\d}{\d t}\frac{\partial L}{\partial y^{A}}+C_{AB}^{D}y^{B}\frac{\partial L}{\partial y^{D}}=\frac{\partial L}{\partial s}\frac{\partial L}{\partial y^{A}}, \\
			& \frac{\d s}{\d t}=L(y^{A},s),
		\end{split}		
	\end{equation*}
    for the Lagrangian  $L:\mathfrak{g}\times \R \rightarrow \R$ and structure constants $C_{AB}^{D}$.
}\end{example}

\begin{example}\label{Atiyah:eq}{\rm
Let $\mathcal{A}:TQ\to\al$ be a principal connection in the principal bundle $\pi:Q\to Q/G$ and $\mathcal{B}:TQ\oplus TQ\to\al$ be the curvature of $\mathcal{A}.$ 


We will use coordinates $(q^i,q^A)$ on a suitable open subset $\pi^{-1}(U)$  (containing $U\times \, \{e\}$, where $e$ is the identity of $G$) such that $(q^i)$ are coordinates on $U$, and $(q^A)$  are coordinates on the fibre $G$, where $i= 1,\ldots,n-d=\mathop{\dim}Q-\mathop{\dim}G$, $A= 1,\ldots, d=\mathop{\dim}G$. 
 Then, the local expression of the projection $\pi: Q \to Q/G$ is $\pi(q^i,q^A) = (q^i)$.


Suppose that $ \{e_{A}\}$ is a basis of $\al,$ and denote by $\{\widehat{e_{A}}\}$ the fundamental vector fields on $Q$ given by
$$\widehat{e_{A}}(q,g)=(ad_{g} e_A)_Q (q,g) ,$$ where
$ad_{g}: \al \to \al$ is the adjoint action. If
$$\mathcal{A}\left(\frac{\partial}{\partial
q^{i}}\Big{|}_{(q,e)}\right)=\mathcal{A}_{i}^{A}(q)\, e_{A},\quad\mathcal{B}\left(\frac{\partial}{\partial
q^{i}}\Big{|}_{(q,e)},\frac{\partial}{\partial
q^{j}}\Big{|}_{(q,e)}\right)=\mathcal{B}_{ij}^{A}\,(q)e_{A},$$ 
for $i,j= 1,\ldots, n-d$ and $q\in U,$ then the horizontal lift of the
vector field $\frac{\partial}{\partial q^{i}}$ is the vector field
on $\pi^{-1}(U)\simeq U\times G$ given by 
$$e_i = \left(\frac{\partial}{\partial
q^{i}}\right)^{h}=\frac{\partial}{\partial
q^{i}}-\mathcal{A}_{i}^{A} \, \widehat{e_{A}}.$$

Therefore, the vector fields $e_{i}$, $\widehat{e_{A}}$ on $U\times G$ are $G$-invariant under the action
of $G$ over $Q$ and define a local basis  $\{e_{i},\widehat{e_{A}}\}$ on
$Sec(TQ/G)$ which induces local coordinates $(q^{i},\dot{q}^{i},v^{A})$ on $TQ/G$.

Then, we obtain the Lagrange-Poincaré-Herglotz equations (see \cite{ACLSS-2023}) for $L:\widehat{TQ}\times \R\to\R$ given by
        \begin{align}
            \frac{\partial L}{\partial q^{j}}-\frac{\d}{\d t}\left(\frac{\partial L}{\partial \dot{q}^{j}}\right)&=\frac{\partial L}{\partial v^{A}}\left(\mathcal{B}_{ij}^{A}\dot{q}^{i}+c_{DB}^{A}\mathcal{A}_{j}^{B} v^{B}\right) - \frac{\partial L}{\partial s}\frac{\partial L}{\partial \dot{q}^{j}} \quad\forall j,\nonumber\\
            \frac{\d}{\d t}\left(\frac{\partial L}{\partial v^{B}}\right)&=\frac{\partial L}{\partial v^{A}}\left(c_{DB}^{A} v^{D}-c_{DB}^{A}\mathcal{A}_{i}^{D}\dot{q}^{i}\right) + \frac{\partial L}{\partial s}\frac{\partial L}{\partial v^{B}}\quad\forall B,\\\nonumber
            \frac{\d s}{\d t} &= L , \nonumber
\end{align}
being $\{c_{AB}^{C}\}$ the constant structures of $\al$ with respect to the basis $\{e_{A}\}$ (see \cite{LMM-2005} for more details).

}
\end{example}

\section{Contact Hamiltonian formalism on Lie algebroids}\label{Ha}

In this section, we extend the standard Hamiltonian  contact formalism to Lie algebroids. Let   $(E,\lcf\cdot,\cdot\rcf,\rho)$ be a Lie algebroid of rank $m$ over a manifold $Q$ of dimension $n$ and
$\tau^{\;*}:E^{\;*}\to Q$ be the vector bundle projection of the dual bundle $E^*$ of $E$.

\subsection{The contact Hamiltonian prolongation}\label{Ham prol}

 The standard contact Hamiltonian formalism is developed on the bundle  $T^*Q\times \r$. For this generalization  to Lie algebroids,  it is natural to consider the projection map
$\pi  \colon E^*\times \r\to Q$ given by $\pi  (b_q^* ,s)= {q}$, being now $P= \ers$ and $(b_q^* ,s)$ an element of $E^* \times \r$.

Let $(q^i)$ be local coordinates on a neighborhood $U$ of $Q$, $i=1,\ldots,n$, and
$\{e^\alpha\}$ be a local basis of sections of $\tau^*:E^*\to Q$, $\alpha=1,\ldots,m$. Given ${b^*_q}\in E_q^*$, we can write
$b^*_q=y_\alpha(b^*_q)e^\alpha( {q})\in E^*_ {q} $, so the
coordinates of $b^*_q\in E^*$ are $(q^i({q}),y_\alpha({b^*_q}))$
and each
section $\sigma$ is  given locally by $\sigma\big\vert_{U}=y_\alpha
e^\alpha$.
Then the local coordinates on
$\pi^{-1}(U)\subseteq E^*\times \r$ are given by
\begin{equation}\label{cke}
q^i(b_q^* ,s)=q^i({q}),\quad
y_\alpha (b_q^* ,s)=y_\alpha(b_q^*),\quad s(b_q^* ,s)=s. \end{equation}
 
Consider now the prolongation of $E$ over the fibration $\pi \colon E^*\times \r \to Q$
\begin{equation}\label{tech}\tech=\{(a_{q},v_{(b_q^* ,s)})\in   E\times  T( E^*\times \r) /\;
\rho(a_{q})=T \pi (v_{(b_q^* ,s)})\}\,.
\end{equation}
By \eqref{vicoord}, we know that
$(a_q,v_{(b_q^*,s)})\in \tech $
if, and only if,
$$
v_{(b_q^*,s)}=y^\alpha(a_q) \,\rho^i_\alpha(q)\derpar{}{q^i}+  \dot{y}_\alpha \derpar{}{y_\alpha}       +  \dot{s} \derpar{}{s}  ,   
$$
being $(q^i,y^\alpha,s,z_\alpha,\dot y_\alpha, \dot s )$ the induced local coordinates on $\tech$, where
\begin{equation}\label{local coord cprol h}\begin{array}{lcllcl}
 q^i(a_{q},v_{(b_q^* ,s)}) &=& q^i({q})
  \;,\quad & z_\alpha (a_{q},v_{(b_q^* ,s)})&=&y_\alpha (b_q^*) \;,\\\noalign{\medskip}
y^\alpha(a_{q},v_{(b_q^* ,s)})&=&y^\alpha(a_{q})
  \;,\quad  & \dot y_\alpha(a_{q},v_{(b_q^* ,s)})
  &=&v_{(b_q^* ,s) }(y_\alpha )
 \;,\\\noalign{\medskip}
 s(a_{q},v_{(b_q^* ,s)}) &=& s\;,
\quad &  \dot s(a_{q},v_{(b_q^* ,s)})&=&v_{(b_q^* ,s)}(s)\;. \\
\end{array}
\end{equation}

 From Section \ref{Sec 5.2.}, we deduce the following properties of $\tech$.
\begin{enumerate}
\item The vector bundle $\tech$ with projection $
\widetilde{\tau}_{E^*\times \r}\colon  \tech \to E^*\times \r$ given by $\widetilde{\tau}_{E^*\times \r}(a_{q},v_{(b_q^* ,s)})=(b_q^* ,s)
$ has a Lie algebroid structure
$(\lcf\cdot,\cdot\rcf^{*\pi },\rho^{*\pi }\,)$,
where the anchor map
$\rho^{*\pi }\colon  \tech\to
T( E^*\times \r)$ given by $ \rho^{*\pi }( (a_{q},v_{(b_q^* ,s)}))=v_{(b_q^* ,s)} $ is the canonical projection on the second factor. We refer to this Lie algebroid as the {\it contact Hamiltonian prolongation}.

The following diagram shows the different projections defined from $\tech$
$$\xymatrix@=5mm{ \mathcal{T}^E(E^*\times \r)\ar[d]_-{   \tau_1}\ar[r]^-{\rho^{*\pi}}\ar@/^{10mm}/[rr]^-{   \widetilde\tau_{E^*\times \r}}
	& T( E^*\times \r)\ar[r]^-{\tau_{E^*\times \r}}\ar[d]^-{T\pi} & E^* \times \r\ar[d]^-{\pi}
	\\
	E\ar[r]^-{\rho} & T Q\ar[r]^-{\tau_Q} & Q }$$
where
\begin{equation}\label{projection prol}
	\hspace{-0.4cm}\begin{array}{lclclclclcl}
		\tau_1(a_{q},v_{(b_q^* ,s)})&=&a_q &, &
		\rho^{*\pi}(a_{q},v_{(b_q^* ,s)})&=&v_{(b_q^* ,s)}&, &
		\widetilde{\tau}_{E^*\times \r}(a_q,v_{(b_q^*,s)})&=& (b_q^*,s).
\end{array}\end{equation}

\item  The set  $\widetilde{\mathcal{X}}_\alpha,\widetilde{\mathcal{V}}_\alpha,\widetilde{\mathcal{V}}_s\colon  E^*\times \r\to\mathcal{T}^E(E^*\times \r)$
given by
{\small\begin{equation}\label{basehc} \begin{array}{c}
\widetilde{\mathcal{X}}_\alpha(b_q^* ,s) =
\left(e_\alpha({q}), \rho^i_\alpha({q})\displaystyle\frac{\partial
}{\partial q^i}\Big\vert_{(b_q^* ,s)}\right)\, ,  \,\,
\widetilde{\mathcal{V}}_\alpha(b_q^* ,s)
=\left(0_{q}, \displaystyle\frac{\partial}{\partial
y_\alpha }\Big\vert_{(b_q^* ,s)}\right) \, ,  \,\, \widetilde{\mathcal{V}}_s(b_q^* ,s)=
\left(0_{q}, \displaystyle\frac{\partial}{\partial
s}\Big\vert_{(b_q^* ,s)}\right)
\end{array}\end{equation}} 
is a local basis of  $Sec(\tech)$, the set of sections of  $\widetilde{\tau}_{E^*\times \r}$ (see \eqref{base k-prol}).

\item The anchor map $\rho^{*\pi }\colon  \tech\to
T( E^*\times \r)$ allows us to associate a vector field with each section $\xi\colon E^*\times \r\to\tech$ of $\widetilde{\tau}_{E^*\times \r}$. Locally, if $\xi$ is given by
$$\xi=\xi_2^\alpha\widetilde{\mathcal{X}}_\alpha+\xi_1^\alpha  \widetilde{\mathcal{V}}_\alpha + \xi_0\widetilde{\mathcal{V}}_s \in Sec(\tech),$$
then the associate vector field is
\begin{equation}
\label{rholksimhc}
\rho^{*\pi }(\xi)=\rho^i_\alpha
\xi_2^\alpha\derpar{}{q^i} + \xi_1^\alpha\derpar{}{y_\alpha }+ \xi_0\derpar{}{s} \in
\vf(E^*\times \r)\,.\end{equation}

\item The Lie bracket of two sections of $\widetilde{\tau}_{E^*\times \r}$ is characterized by the relations (see \eqref{lie brack k-prol}),
\begin{equation}\label{lie brack tech}
\begin{array}{lclclclclcl}
\lcf\widetilde{\mathcal{X}}_\alpha,\widetilde{\mathcal{X}}_\beta\rcf^{*\pi }&=&
\mathcal{C}^\gamma_{\alpha \beta}\widetilde{\mathcal{X}}_\gamma\;,
&\quad &
\lcf\widetilde{\mathcal{X}}_\alpha,\widetilde{\mathcal{V}}_\beta\rcf^{*\pi }&=&0 \;, &\quad &
\lcf\widetilde{\mathcal{X}}_\alpha,\widetilde{\mathcal{V}}_s\rcf^{*\pi }&=&
0\;,
\\\noalign{\medskip}
\lcf\widetilde{\mathcal{V}}_\alpha,\widetilde{\mathcal{V}}_\beta\rcf^{*\pi }&=&0
\;, &\quad &
\lcf\widetilde{\mathcal{V}}_\alpha,\widetilde{\mathcal{V}}_s\rcf^{*\pi }&=&
0 \;,
&\quad &
\lcf\widetilde{\mathcal{V}}_s,\widetilde{\mathcal{V}}_s\rcf^{*\pi }&=& 0 \,.
\end{array}
\end{equation}

\bigskip

\item If $\{\widetilde{\mathcal{X}}^\alpha,
\widetilde{\mathcal{V}}^\alpha,\widetilde{\mathcal{V}}^s\}$ is the dual basis of
$\{\widetilde{\mathcal{X}}_\alpha,\widetilde{\mathcal{V}}_\alpha,
\widetilde{\mathcal{V}}_s\}$, then the exterior differential is given by
\begin{equation}\label{ext dif tech}
\begin{array}{l}
\d^{\tech}f= \rho^i_\alpha\derpar{f}{q^i}\widetilde{\mathcal{X}}^\alpha
+ \derpar{f}{y_\alpha }\widetilde{\mathcal{V}}^\alpha + \derpar{f}{s  }\widetilde{\mathcal{V}}^s \,,\quad \makebox{ for
all }  \; f \in \mathcal{C}^\infty(E^*\times \r)\\\noalign{\medskip}
\d^{\tech}\widetilde{\mathcal{X}}^\gamma =
-\displaystyle\frac{1}{2}\mathcal{C}^\gamma_{\alpha\beta}\widetilde{\mathcal{X}}^\alpha\wedge\widetilde{\mathcal{X}}^\beta
\quad , \quad
\d^{\tech}\widetilde{\mathcal{V}}^\gamma   =0 \quad , \quad \d^{\tech}\widetilde{\mathcal{V}}^s   =0\,.
\end{array}
\end{equation}
From now on we are going to set the notation $\d = \d^{\tech}$.
\end{enumerate}

\begin{remark}\label{remark equiv h}
Note that in the particular case $E= T Q$, the manifold $\mathcal{T}^E(E^*\times \r)$ reduces to $T(T^*Q\times \r)$. 
\end{remark}

\subsection{Hamiltonian formalism}\

The Liouville $1$-section $\Theta : \ers \to (\mathcal{T}^E(E^*\times \r))^{\,*}$ is defined by
 \begin{equation}\label{theta A kc}
 \Theta_{(b_q^*,s) }( a_{q}, v_{(b_q^*,s) })=b_q^*(a_{q})\;,
\end{equation}
for each $ a_{q}\in E,\,(b_q^*,s)  \in E^*\times \r$ and $ v_{(b_q^*,s) }\in T_{(b_q^*,s) }(E^*\times \r)$. 

Now, we define the following $1$-section 
$\eta$ on $(\mathcal{T}^E(E^*\times \r))^{*}$  as
\begin{equation}\label{contsec}
 \eta = \widetilde{\mathcal{V}}^s - \Theta ,   
\end{equation}
and its differential
$\d \eta:E^*\times \r \to
\Lambda^2(\mathcal{T}^E(E^*\times \r))^{*}   $
satisfies
$\d \eta  = - \d\Theta $.

From \eqref{basehc} we deduce that the local expressions of $\Theta$ and $\eta$ are
\begin{equation}\label{theta*kcsim}
\Theta=y  _\alpha\widetilde{\mathcal{X}}^\alpha \;,\qquad \eta = \widetilde{\mathcal{V}}^s - y_\alpha \widetilde{\mathcal{X}}^\alpha ,
\end{equation}
and from the local expressions \eqref{ext dif tech} and \eqref{theta*kcsim}, we obtain
\begin{equation}\label{local deta}
\d \eta =  \displaystyle\frac{1}{2} \,
\mathcal{C}^\gamma_{\alpha\beta} \,
	y_\gamma \, \widetilde{\mathcal{X}}^\alpha \wedge \widetilde{\mathcal{X}}^\beta
+ \widetilde{\mathcal{X}}^\gamma \wedge \widetilde{\mathcal{V}}^\gamma \,.
\end{equation}
\begin{remark}\label{prolongcontalg}
From \eqref{theta*kcsim} and \eqref{local deta} we deduce that  $\eta$ defines a contact structure of the Lie algebroid $(\tech,\lcf\cdot,\cdot\rcf^{*\pi },\rho^{*\pi }\,)$ in the sense of Definition \ref{defcontLiealg}. Moreover, the Reeb section $\Rb$ for this contact Lie algebroid, characterized by \eqref{Reebsec-condition}, is locally given by
$\Rb = \widetilde{\mathcal{V}}_s$.
\end{remark}

\begin{remark}
When $E=T Q$ and $\rho= id_{T Q}$, $\eta$ is the canonical contact structure \eqref{caneta}.
\end{remark}

\paragraph{\bf The contact Hamilton equations.}\label{Sec 8.2.2.}\

\begin{theorem}
Let $H:E^*\times \r\to \r$ be a Hamiltonian function.
Then, since $\eta$ is a contact section of $(\tech,\lcf\cdot,\cdot\rcf^{*\pi },\rho^{*\pi }\,)$, there exists a unique section 
$\xi_H :E^* \times \r \to
\mathcal{T}^E(E^*\times \r)$ of
$\widetilde{\tau}_{E^*\times \r}$, called the Hamiltonian section, satisfying
\begin{equation}\label{geometric ec H al}
\imath_{\xi_H} \eta =-H\quad,\quad
\displaystyle \imath_{\xi_H}\d \eta   =\d H-
\rho^{*\pi }(\Rb)(H) \eta \,.
\end{equation}
Moreover, if
$\widetilde{c}:\r\to E^*\times \r,\,\, \widetilde{c}(t) =
(   c^i( t),c_\alpha( t ),  c_s(t))$  is an integral curve of $\xi_H$, then $\widetilde{c}$ is a solution of the following system of differential equations
\begin{equation} \label{Hamilton eq}
\begin{array}{rcl}\displaystyle\frac{\d c^{i}}{\d t}\Big\vert_{ t }
&=&\rho^i_\alpha\,\displaystyle\frac{\partial H}{\partial
y_\alpha}\Big\vert_{\widetilde{c}( t )}\;,
\\\noalign{\medskip}
 \displaystyle\frac{\d c_\alpha}{\d t  }\Big\vert_{ t } &=& -
\Big(
 \rho^i_\alpha\,\displaystyle\frac{\partial H}{\partial
q^i}\Big\vert_{\widetilde{c}( t )}+ \mathcal{C}^\gamma_{\alpha\beta}\,
c_\gamma\,
\displaystyle\frac{\partial H}{\partial
y  _\beta}\Big\vert_{\widetilde{c}( t )} + c_\alpha \, \ds\frac{\partial H}{\partial s}\Big\vert_{\widetilde{c}( t )}\Big)\,,
\\\noalign{\medskip}
\displaystyle\frac{\d   c_s}{\d t  }\Big\vert_{ t } &=& c_\alpha\, \ds\frac{\partial H}{\partial y_\alpha}\Big\vert_{\widetilde{c}( t )} - H(\widetilde{c}(t))\,.
\end{array}
\end{equation}
These equations are called the {\bf contact Hamilton equations} on Lie algebroids.
\end{theorem}

\proof
Proceeding in the same way as in the Lagrangian case (see Theorem \ref{algeform co}), we obtain from \eqref{ext dif tech}, \eqref{theta*kcsim}, \eqref{local deta}  and \eqref{geometric ec H al}, the local expression of
$\xi_H$ 
\begin{equation}\label{explocxi}
\xi_H = \displaystyle\frac{\partial H}{\partial y  _\alpha}
\widetilde{\mathcal{X}}_\alpha - \left(
\rho^i_\alpha\,\displaystyle\frac{\partial H}{\partial q^i}+
\mathcal{C}^\gamma_{\alpha\beta}\,y_\gamma\, \displaystyle\frac{\partial
H}{\partial y  _\beta} + y_\alpha \, \ds\frac{\partial H}{\partial s}\right) \widetilde{\mathcal{V}}_\alpha+ \left(y_\alpha\, \ds\frac{\partial H}{\partial y_\alpha} - H \right) \widetilde{\mathcal{V}}_s .
\end{equation}
Then, an integral curve $\widetilde{c}(t)$ of $\xi_H$, that is, an integral curve of the vector field $\rho^{*\pi}(\xi_H)$, is a solution of \eqref{Hamilton eq}.
\qed

\begin{remark}
In the particular case $E=T Q$ and $\rho=id_{T Q}$,  equations \eqref{Hamilton eq} are the standard contact Hamilton equations \eqref{eq:contact-hamiltonian-equations-darboux-coordinates}.
\end{remark}

In addition to the Hamiltonian section $\xi_H$ associated to a Hamiltonian function $H:E^*\times \r\to \r$, there is another relevant section, called the \textit{evolution section} $ \E_{v_H} \in Sec(\tech)$ defined by 
$$\E_{v_H} = \xi_H + H \Rb ,$$
so that it reads in local coordinates as follows
\begin{equation}\label{evloc}
\E_{v_H} = \displaystyle\frac{\partial H}{\partial y  _\alpha}
\widetilde{\mathcal{X}}_\alpha - \left(
\rho^i_\alpha\,\displaystyle\frac{\partial H}{\partial q^i}+
\mathcal{C}^\gamma_{\alpha\beta}\,y_\gamma\, \displaystyle\frac{\partial
H}{\partial y  _\beta} + y_\alpha \, \ds\frac{\partial H}{\partial s}\right) \widetilde{\mathcal{V}}_\alpha+ y_\alpha\, \ds\frac{\partial H}{\partial y_\alpha} \widetilde{\mathcal{V}}_s .
\end{equation}
Then, the integral curves of $\E_{v_H}$ satisfy
\begin{equation}
\ds\frac{\d q^{i}}{\d t} = \rho^i_\alpha\,\displaystyle\frac{\partial H}{\partial
	y  _\alpha}\,,\qquad
\ds\frac{\d y_\alpha}{\d t} = -
\left(
\rho^i_\alpha\,\displaystyle\frac{\partial H}{\partial
	q^i}+ \mathcal{C}^\gamma_{\alpha\beta}\,
y_\gamma\,
\displaystyle\frac{\partial H}{\partial
	y  _\beta} + y_\alpha \, \ds\frac{\partial H}{\partial s}\right)\,,\qquad
	\ds\frac{\d s}{\d t} = y_\alpha\, \ds\frac{\partial H}{\partial y_\alpha}\,.
\end{equation}

\begin{remark}
   Recently, we have introduced a Jacobi structure on $E^{*}\times \R$ in order to deduce the contact equations of motion on Lie algebroids (see \cite{ACLSS-2023}). In this framework, we make no use of a contact structure on Lie algebroids but only the Jacobi structure. Let $X_{H}\in \mathfrak{X}(E^{*}\times \R)$ be the Hamiltonian vector field obtained from the Hamiltonian function $H:E^{*}\times \R\rightarrow \R$ using the above mentioned Jacobi structure. Comparing the local expression of the integral curves of both mechanical systems we deduce that $$\rho^{*\pi}\left( \xi_H \right) = X_{H}.$$
\end{remark}

\begin{example}
{\rm 
	When $E=TQ$ is equipped with the Lie brackets and the anchor map is just the identity, then the Jacobi structure is the canonical one in $T^{*}Q\times \R$. In that case, we recover the contact Hamiltonian equations \eqref{eq:contact-hamiltonian-equations-darboux-coordinates}
$$\ds\frac{\d q^i}{\d t}=\frac{\partial H}{\partial p_{i}},\,\qquad \ds\frac{\d p_i}{\d t}=\frac{\partial H}{\partial q^{i}}-p_{i}\frac{\partial H}{\partial s},\,\qquad \ds\frac{\d s}{\d t}=p_{i}\frac{\partial H}{\partial p_{i}}-H.$$ 
}
\end{example}

\begin{example}
{\rm 
	When $E$ is a Lie algebra, say $E=\mathfrak{g}$, considering adapted coordinates $(p_{A}, s)$ to a dual basis of the Lie algebra $\{e^{A}\}$, we find that the Hamiltonian vector field on $\mathfrak{g}^{*}\times \R$ is just
	\begin{equation}
		X_{H}=-\left( C_{A B}^{D} p_{D}\frac{\partial H}{\partial p_{B}}+ p_{A}\frac{\partial H}{\partial s} \right)\frac{\partial}{\partial p_{A}}+\left( p_{A}\frac{\partial H}{\partial p_{A}}-H(p_{A}, s) \right)\frac{\partial}{\partial s}.
	\end{equation} which gives rise to the Lie-Poisson-Jacobi equations (see \cite{ACLMMP-2023})

 $$\ds\frac{\d p_A}{\d t}=-C_{A B}^{D} p_{D}\frac{\partial H}{\partial p_{B}}- p_{A}\frac{\partial H}{\partial s},\,\qquad \ds\frac{\d s}{\d t}=p_{A}\frac{\partial H}{\partial p_{A}}-H(p_{A}, s).$$
 }
\end{example}

\begin{example}
{\rm 
Given a Hamiltonian function $H:T^{*}Q/G \times \R \rightarrow \R$ associated with the Atiyah algebroid $TQ/G\to Q/G$, let $\{e_i, \widehat{e_{A}}\}$ be the local basis of $G$-invariant vector fields on $Q$ given in Example \ref{Atiyah:eq}, and $(q^i,\dot{q}^i,v^A)$ be the corresponding local fibred coordinates on $TQ/G$. Then, denote by $(q^i,p_i,\bar{p}_A)$ the (dual)
coordinates on $T^*Q/G$ and $(q^i,p_i,\bar{p}_A,s)$ the corresponding coordinates on $T^*Q/G \times \R$.

In these coordinates, the contact Hamiltonian equations are given by the Hamilton-Poincar\'e-Herglotz equations (see \cite{ACLSS-2023})
\begin{equation*}
    \begin{split}
        & \ds\frac{\d q^i}{\d t} = \frac{\partial H}{\partial p_{i}}, \, \qquad \ds\frac{\d p_i}{\d t} = -\frac{\partial H}{\partial q^{i}} + \mathcal{B}_{ij}^A \bar{p}_A \frac{\partial H}{\partial p_{j}} - c_{AB}^{C}\mathcal{A}_{i}^{B}\bar{p}_C \frac{\partial H}{\partial \bar{p}_{A}} - p_{i}\frac{\partial H}{\partial s} , \\
        & \ds\frac{\d \bar{p}_A }{\d t} = c_{AB}^{C}\mathcal{A}_{i}^{B}\bar{p}_C \frac{\partial H}{\partial p_{i}} - c_{AB}^{C} \bar{p}_C \frac{\partial H}{\partial \bar{p}_{B}} - \bar{p}_{A}\frac{\partial H}{\partial s}, \, \qquad \ds\frac{\d s}{\d t}= p_{i}\frac{\partial H}{\partial p_{i}} + \bar{p}_{A}\frac{\partial H}{\partial \bar{p}_{A}} - H.
    \end{split}
\end{equation*}

}
\end{example}

\subsection{The Legendre transformation and the equivalence between the Lagrangian and Hamiltonian formalisms}

Let $L:E \times \r\to\r$ be  a Lagrangian function. We introduce the {\it Legendre transformation} associated to $L$ as the map defined by
$$
\begin{array}{rl}
\Lec_L : & E \times \r \longrightarrow E^*\times \r    \\ \noalign{\medskip}
     & (a_q,s) \longmapsto \Lec_L(a_q,s)=(\mu_q(a_q,s) ,s) ,
\end{array}
$$  
where
$$\mu_q(a_q,s): E_q\to\r \,, \qquad \mu_q(a_q,s)(b_q)= \displaystyle\frac{d}{dt}\Big\vert_{t=0}L(
a_{q}+tu_{q},s)  $$
 being $b_q\in
E_{q}$.

The map $\Lec_L$ is well defined and its local expression in fibred coordinates $(q^i,y^\alpha,s)$ on $E\times \r$, and $(q^i,y_\alpha,s)$ on $E^* \times \r$ is
\begin{equation}\label{locLeg}
Leg_L(q^i,y^\alpha,s)=\left(q^i,\displaystyle\frac{\partial L}{\partial y^\alpha},s\right) .
\end{equation}

From this local expression it is easy to prove that the Lagrangian $L$ is regular if, and only if, $\Lec_L$ is a local diffeomorphism.

The Legendre map induces a mapping
$\mathcal{T}^E \Lec_L :\tec \to
\mathcal{ T}^E( E^*\times \r)$ defined by
\begin{equation}\label{defLegprol}
\mathcal{T}^E {\Lec}_L (a_{q},v_{(b_q,s)}) =
\left(a_{q},({\Lec}_L)_*(b_q,s  )(v_{(b_q,s) })\right) ,  
\end{equation}
where $a_q\in
E_{q},\;(b_q,s  )\in E\times \r $ . 

Using \eqref{locLeg}, we deduce that the local expression of $\mathcal{T}^E \Lec_L$ in the coordinates of $\tec$ and $\tech$ (see Sections \ref{lagran prolong} and \ref{Ham prol}) is
\begin{equation}\label{locLegprol}
\mathcal{T}^E\Lec(q^i,y^\alpha,s, z^\alpha, 
\dot y^\alpha, \dot s)=\left(q^i, \displaystyle\frac{\partial L}{\partial y^\alpha},s, z^\alpha, \rho^i_\beta \, y^\beta \ds\frac{\partial^2 L}{\partial q^i \partial y^\alpha}+ \dot{y}^\beta    \ds\frac{\partial^2 L}{\partial y^\beta \partial y^\alpha} + \dot{s} \ds\frac{\partial^2 L}{\partial s \partial y^\alpha},\dot{s}\right)\,.
\end{equation}

\begin{theorem}\label{equivforma al}
Let $L\colon E \times \r\to\r$ be a regular Lagrangian. The pair $(\mathcal{T}^E\Lec_L, \Lec_L)$ is a morphism between the Lie algebroids $(\tec,\lcf\cdot,\cdot\rcf^{\pi },\rho^{\pi }\,)$
and
$(\tech,\lcf\cdot,\cdot\rcf^{* \pi },\rho^{* \pi })$
\[\xymatrix@R=8mm{\mathcal{T}^E(E \times \r)
\ar[rr]^-{\mathcal{T}^E\Lec_L}\ar[d]_-{\widetilde{\tau}_{E \times \r}}
&&\mathcal{T}^E(E^*\times \r)\ar[d]^-{\widetilde{\tau}_{E^* \times \r}}\\
E\times \r\ar[rr]^-{\Lec_L}
&&E^*\times \r}\] 
Moreover, if  $\eta_L$   (respectively, $\eta$) is the Lagrangian contact section associated to
$L$ (respectively, the contact section on
$(\tech)^*$), then
\begin{equation}\label{equiv formas}
(\mathcal{T}^E {\Lec}_L, {\Lec}_L)^*\eta=\eta_L \,, \qquad (\mathcal{T}^E {\Lec}_L, {\Lec}_L)^* \left(\d^{\tech}\eta\right)= \d^{\tec} \eta_L .
 \end{equation}
\end{theorem}

\begin{proof}
First, we have to prove that the pair $(\mathcal{T}^E\Lec_L, \Lec_L)$ satisfies the condition \eqref{lie morph} to be a Lie algebroid morphism.

Let $(q^i)$ be local coordinates on $Q$, $\{e_\alpha\}$ a local basis of sections of $E$, and $\{\mathcal{X}_\alpha,\mathcal{V}_\alpha, \mathcal{V}_s\}$ and $\{\widetilde{\mathcal{X}}_\alpha,\widetilde{\mathcal{V}}_\alpha, \widetilde{\mathcal{V}}_s\}$ the corresponding local basis of sections of $\tec$ and $\tech$, respectively. Then, using \eqref{difksimc}, \eqref{defLegprol}  and \eqref{locLegprol} we deduce that
{\small\begin{equation}
\left(\mathcal{T}^E {\Lec}_L, {\Lec}_L\right)^* \widetilde{\mathcal{X}}^\alpha = \mathcal{X}^\alpha \,, \quad \left(\mathcal{T}^E {\Lec}_L, {\Lec}_L\right)^* \widetilde{\mathcal{V}}^\alpha = \d^{\tec} \left(\ds\frac{\partial L}{\partial y^\alpha}\right) \,, \quad \left(\mathcal{T}^E {\Lec}_L, {\Lec}_L\right)^* \widetilde{\mathcal{V}}^s = \mathcal{V}^s
\end{equation}}
Thus, from \eqref{difksimc} and \eqref{ext dif tech} we conclude
\begin{align}
& \left(\mathcal{T}^E {\Lec}_L, {\Lec}_L\right)^*\left(\d^{\tech} f' \right)  = \d^{\tec} \left( f' \circ {\Lec}_L \right) , \label{pulldifunc} \\  \noalign{\medskip}
& \left(\mathcal{T}^E {\Lec}_L, {\Lec}_L\right)^*\left(\d^{\tech} \widetilde{\mathcal{X}}^\alpha \right)  = \d^{\tec} \left( \left(\mathcal{T}^E {\Lec}_L, {\Lec}_L\right)^* \widetilde{\mathcal{X}}^\alpha \right) , \\ \noalign{\medskip}
& \left(\mathcal{T}^E {\Lec}_L, {\Lec}_L\right)^*\left(\d^{\tech} \widetilde{\mathcal{V}}^\alpha \right)  = \d^{\tec} \left( \left(\mathcal{T}^E {\Lec}_L, {\Lec}_L\right)^* \widetilde{\mathcal{V}}^\alpha \right) , \\ \noalign{\medskip}
& \left(\mathcal{T}^E {\Lec}_L, {\Lec}_L\right)^*\left(\d^{\tech} \widetilde{\mathcal{V}}^s \right)  = \d^{\tec} \left( \left(\mathcal{T}^E {\Lec}_L, {\Lec}_L\right)^* \widetilde{\mathcal{V}}^s \right) ,
\end{align}
for all $f' \in \Cinfty(\ers)$ and for all $\alpha$, which proves that the pair $(\mathcal{T}^E\Lec_L, \Lec_L)$ is a Lie algebroid morphism.

Finally, from \eqref{locLeg} and \eqref{defLegprol}, using the local expressions \eqref{local thetac} and \eqref{theta*kcsim}  of $\eta_L$ and $\eta$ respectively, and taking into account the above results, we deduce \eqref{equiv formas}.
\end{proof}

%
Assume now that $L$ is hyperregular, that is, $\Lec_L$ is a global diffeomorphism. From \eqref{defLegprol} and Theorem \ref{equivforma al}, we deduce that the pair $(\mathcal{T}^E {\Lec}_L, {\Lec}_L)$ is a Lie algebroid isomorphism. Moreover, we may consider the Hamiltonian function $H: \ers \to \r$ defined by
$$H = E_L \circ {\Lec}_L^{-1} ,$$
where $E_L : \er \to \r$ is the Lagrangian energy associated to $L$ given by \eqref{energy}. The Hamiltonian section $\xi_H \in Sec(\tech)$ is characterized by the conditions \eqref{geometric ec H al} and the Lagrangian section $\Gamma_L \in Sec(\tec)$ is characterized by \eqref{ec ge EL co}. Therefore, we have the following.
\begin{theorem}\label{sanantonio}
If the Lagrangian $L$ is hyperregular, then the Lagrangian section $\Gamma_L$ associated to $L$  and the Hamiltonian section $\xi_H$ are $(\mathcal{T}^E {\Lec}_L, {\Lec}_L)$-related, that is,
\begin{equation}\label{Legrelated}
\xi_H \circ {\Lec}_L = \mathcal{T}^E {\Lec}_L \circ \Gamma_L .
\end{equation}

Moreover, if $\widetilde{c} : I \subset \r \to \er$ is a solution of the Herglotz equations associated to $L$, then $\sigma = {\Lec}_L \circ \widetilde{c} :  I \subset \r \to \ers$ is a solution of the Hamilton equations associated to $H$ and, conversely, if $\sigma :  I \subset \r \to \ers$ is a solution of the Hamilton equations, then $\widetilde{c} = {\Lec}_L^{-1} \circ \sigma$ is a solution of the Herglotz equations.
\end{theorem}

\begin{proof}
Let $\xi_H = \mathcal{T}^E {\Lec}_L \circ \Gamma_L \circ {\Lec}_L^{-1}$ be the Hamiltonian section solution to \eqref{geometric ec H al}. Then, from \eqref{ec ge EL co}, \eqref{equiv formas}, \eqref{pulldifunc}, and since
$$\mathcal{T}^E {\Lec}_L ( \Rb_L) = \Rb ,$$
we obtain that \eqref{Legrelated} holds. Now, using \eqref{Legrelated} and Theorem \ref{equivforma al}, we deduce the second part.
\end{proof}

\begin{remark}
When $E= T Q$, the Legendre transformation defined above coincides with the Legendre map of the standard contact formalism, and Theorem \ref{sanantonio} gives the equivalence between the standard contact Lagrangian and Hamiltonian formalisms, see \cite{LGLRR-2021,Gaset2020}.
\end{remark}

\section{Legendrian Lie subalgebroids in contact Lie algebroids}
\label{legen-lie-subalg-2023}

Given a contact manifold $(M,\eta)$, an interesting type of submanifolds are the so-called Legendrian submanifolds, and several results are known that help us to understand the dynamics of a contact system as a Legendrian submanifold (see for example, \cite{LLM-21}). This concept is the natural extension of that of Lagrangian submanifold that has been extensively used in symplectic geometry \cite{W-1971}, and later generalized to Poisson and Jacobi manifolds \cite{ELLMS-2021, ILMM-1997}.
In order to extend this concept to Lie algebroids, we introduce the notion of a Legendrian Lie subalgebroid of a contact Lie algebroid, and we give a characterization that will allow us to relate these objects with the solution of the Hamilton-Jacobi equation.

\begin{definition}
Let $(E,\lcf\cdot,\cdot\rcf_E, \rho)$ be a contact Lie algebroid of rank $2k+1$ over a manifold $M$ with contact section $\eta$, and $j: F \to E$ , $i: N \to M$ be a Lie subalgebroid (see Definition \ref{subalg}). Then, the Lie subalgebroid is said to be Legendrian if
\begin{enumerate}
    \item $\text{dim} \, F_x = k $,
    \item $\left(\eta(i(x))\right)_{\big\vert_{j(F_x)}} = 0$.
\end{enumerate}
\end{definition}

Let $(E,\lcf\cdot,\cdot\rcf, \rho)$ be a Lie algebroid of rank $m$ over a manifold $Q$ of dimension $n$. Then, the prolongation $\tech$ of $E$ over $\pi: \ers \to Q$ is a contact Lie algebroid (see Theorem \ref{prolongcontalg}). Moreover, if $q$ is a point of $Q$ and $E^*_q \times \r$ is the fibre of $\ers$ over the point $q$, we denote by
$$j_q: T E^*_q \times \r \to  \tech \,, \qquad i_q: E^*_q \times \r \to \ers$$
the maps given by
$$j_q(v_{b^*_q},s) = (0_q,v_{(b^*_q,s)}) \,, \qquad i_q(b_q^*,s) = (b_q^*,s) ,$$
for $(v_{b^*_q},s) \in T E^*_q \times \r$ and $(b^*_q,s) \in E^*_q \times \r$, where $0_{q}: Q \to E$ is the zero section.

On the other hand, if $\gamma$ is a section of $\pi: E^* \times \r \to Q$ we will denote by $F_\gamma$ the vector bundle over $\gamma(Q)$ given by
\begin{equation}\label{Falpha}
 F_\gamma= \{\left(a, T \gamma(\rho(a)\right) \in E \times T( E^* \times \r) / a \in E\}  , 
\end{equation}
and by $j_\gamma : F_\gamma \to \tech$ and $i_\gamma : \gamma(Q) \to E^* \times \r$ the canonical inclusions. Note that the vector bundles $E$ and $F_\gamma$ have the same rank $m$, so that the pair $\left[(\Id_E, T \gamma \circ \rho),\gamma\right]$ is an isomorphism between these vector bundles, where the map $(\Id_E, T \gamma \circ \rho)$ is given by
$$(\Id_E, T \gamma \circ \rho)(a) = (a, T \gamma(\rho(a))) \,, \qquad a \in E .$$
Thus, $F_\gamma$ is a Lie algebroid over $\gamma(Q)$.


\begin{definition}
Consider a function $f : Q \to \r$. We denote by $j^1 f : Q \to E^* \times \r$, $j^1 f = (\d^E f, f)$ the {\it $1$-jet} of $f$, given in local coordinates by
\begin{equation}
j^1 f(q^i) = \left( q^i, \rho^i_\alpha \ds\frac{\partial f}{\partial q^i}, f(q^i)\right).
\end{equation}
\end{definition}
Then, we have the following
\begin{proposition}\label{propLeg}
Let $(E,\lcf\cdot,\cdot\rcf, \rho)$ be a Lie algebroid of rank $m$ over a manifold $Q$ of dimension $n$, and $\tech$ the prolongation of $E$ over $\pi: \ers \to Q$.
\begin{enumerate}
    \item If $q \in Q$, then $j_q : T E^*_q \times \r \to  \tech$ and $i_q: E^*_q \times \r \to \ers$ is a Legendrian Lie subalgebroid of the contact Lie algebroid $\tech$.
    \item If $\gamma \in Sec(\ers)$, then $j_\gamma : F_\gamma \to \tech$ and $i_\gamma : \gamma(Q) \to E^* \times \r$ is a Legendrian Lie subalgebroid of the contact Lie algebroid $\tech$ if, and only if, $\gamma$ is locally the $1$-jet of a function on $Q$.
\end{enumerate}
\end{proposition}

\begin{proof}
 \begin{enumerate}
     \item We can see that the rank of the vector bundle $T E^*_q \times \r \to E^*_q \times \r$ is $m$.
     Consider local coordinates $(q^i)$ on $Q$, $\{e_\alpha\}$ a local basis of sections of $E$, $(q^i,y_\alpha,s)$ the corresponding local coordinates on $\ers$ and $\{\widetilde{\mathcal{X}}_\alpha,\widetilde{\mathcal{V}}_\alpha,\widetilde{\mathcal{V}}_s\}$ the corresponding local basis of sections of $\tech$. From \eqref{basehc}, it follows that
     $$(j_q,i_q)^* (\widetilde{\mathcal{X}}^\alpha) = 0 \, , \qquad (j_q,i_q)^* (\widetilde{\mathcal{V}}^\alpha) = \d^{T E^*_q \times \r} (y_\alpha \circ i_q)  \,, \qquad (j_q,i_q)^* (\widetilde{\mathcal{V}}^s) = 0 \,,$$
and
\begin{equation}\label{jqproof}
 j_q\left(\ds\frac{\partial}{\partial y_\alpha}\Big\vert_{b^*_q}, s \right) = \widetilde{\mathcal{V}}_\alpha(b^*_q,s) \,, \qquad b^*_q \in E^*_q .   
\end{equation}
Using \eqref{ext dif tech}, this implies that $j_q: T E^*_q \times \r \to  \tech$ , $i_q: E^*_q \times \r \to \ers$ is a morphism of Lie algebroids. Thus, since $j_q$ is injective and $i_q$ is an injective immersion, we deduce that $(j_q,i_q)$ is a Lie subalgebroid of $\tech$. Finally, from \eqref{theta*kcsim} and \eqref{jqproof} we conclude that
$$\eta\left(i_q(b^*_q,s)\right)\Big\vert_{j_q\left(T_{b^*_q} (E^*_q) \times \r\right)} = 0 .$$
\item If $\gamma$ is a section of $\ers$ then the Lie algebroids $E \to Q$ and $F_\gamma \to \gamma(Q)$ are isomorphic under $\left[(\Id_E, T \gamma \circ \rho),\gamma\right]$. Note that $(\Id_E, T \gamma \circ \rho)$ is injective and $\gamma$ is an injective immersion. 

On the other hand, from Proposition \ref{propgamma}, we have
$$\left[\, (\Id_E, T \gamma \circ \rho),\gamma \, \right]^* \eta = \d^E   \gamma_s - \gamma_0 ,
$$
where $\gamma = (\gamma_0,   \gamma_s)$, $\gamma_0 : Q \to E^*$ and $\gamma_s : Q \to \r$.

Therefore, the Lie subalgebroid $(j_\gamma,i_\gamma)$ is Legendrian if, and only if, $\gamma$ is locally the $1$-jet of a function on $Q$, namely $\gamma= j^1 \gamma_s$.
\end{enumerate}
\end{proof}

\begin{remark}
When $E$ is the standard Lie algebroid $T Q$, a section $\gamma : Q \to T^* Q \times \r$ is a Legendrian submanifold of $(T^* Q \times \r, \eta_Q)$ if, and only if, $\gamma$ is locally the $1$-jet of a function on $Q$ in the usual sense (see Proposition 3 in \cite{LLM-21}).
\end{remark}

\section{The Hamilton-Jacobi equations}\label{hje}

The Hamilton–Jacobi problem consists in finding a function 
$S:Q\to \r$   (called the
generating function) solution to the equation $H(q^i,\parder{S}{q^i})=E$,
for some $E\in \r$, which is called the Hamilton–Jacobi equation for $H$.
Of course, one can easily see that the above equation  can be written as $\d(H\circ \d S)=0$,
which opens the possibility to consider  general $1$-forms instead of just differentials of a function.


Given a Hamiltonian function $H : E^* \times \r \to \r$, in this section we provide some ingredients necessary to study the Hamilton-Jacobi problem for a contact Hamiltonian section and for the corresponding evolution section.

Let $\tech$ be the prolongation of  the Lie algebroid $(E,\lcf\cdot,\cdot\rcf, \rho)$ over  $\pi: E^* \times \r \to Q$ and consider the morphism $((\Id_E, T \gamma \circ \rho), \gamma))$ between the vector bundles $E$ and $\tech$
$$\xymatrix@C=25mm @R=8mm{
E \ar[r]^{(\Id_E, T \gamma \circ \rho)} \ar[d] & \tech \ar[d] \\
Q \ar[r]^\gamma & \ers
}$$
defined by $(\Id_E, T \gamma \circ \rho)(a_q) = (a_q, (T_q \gamma) \rho(a_q))$, for $a_q \in E_q$ and $q \in Q$.
\begin{proposition}\label{propgamma}
Let $\eta$ be the contact $1$-section of $\tech$ defined on \eqref{contsec}.
If $\gamma$ is a section of $E^* \times \r \to Q$, then the pair $$\left[\, (\Id_E, T \gamma \circ \rho), \gamma \,  \right]$$ is a morphism between the Lie algebroids $(E,\lcf\cdot,\cdot\rcf, \rho)$ and $(\tech,\lcf\cdot,\cdot\rcf^{*\pi}, \rho^{*\pi})$. Moreover,
\begin{equation}
\left[\, (\Id_E, T \gamma \circ \rho),\gamma \, \right]^* \eta = \d^E   \gamma_s - \gamma_0 ,
\end{equation}
where $\gamma = (\gamma_0,   \gamma_s)$, $\gamma_0 : Q \to E^*$ and $\gamma_s : Q \to \r$.
\end{proposition}

\begin{proof}
Consider $(q^i)$ the local coordinates on $Q$ and $\{e_\alpha\}$ the local basis of sections of $E$. Let $\{\widetilde{\mathcal{X}}_\alpha,\widetilde{\mathcal{V}}_\alpha,\widetilde{\mathcal{V}}_s\}$ be a local basis of sections of $\tech$ and
$\{\widetilde{\mathcal{X}}^\alpha,\widetilde{\mathcal{V}}^\alpha,\widetilde{\mathcal{V}}^s   \}$ be the dual basis. Suppose $\gamma$ is locally written as $\gamma(q^i)= (q^i, \gamma_\alpha(q^i),   \gamma_s(q^i))$, then using \eqref{basehc}, it follows that
\begin{equation}
(\Id_E, T \gamma \circ \rho) \circ e_\alpha = \left( \widetilde{\mathcal{X}}_\alpha + \rho^i_\alpha \ds\frac{\partial \gamma_\beta}{\partial q^i} \widetilde{\mathcal{V}}_\beta + \rho^i_\alpha \ds\frac{\partial   \gamma_s}{\partial q^i} \widetilde{\mathcal{V}}_s \right) \circ \gamma ,
\end{equation}
for $\alpha=1,\ldots,m$. 
Thus, from \eqref{defunc} and \eqref{pullsec}  we have
\begin{equation}\label{pullproof}
\left[\,(\Id_E, T \gamma \circ \rho),\gamma\, \right]^* \widetilde{\mathcal{X}}^\alpha = e^\alpha \, , \qquad \left[\,(\Id_E, T \gamma \circ \rho),\gamma\, \right]^* \widetilde{\mathcal{V}}^\alpha = \d^E \gamma_\beta \, , \qquad\left[\,(\Id_E, T \gamma \circ \rho),\gamma\, \right]^* \widetilde{\mathcal{V}}^s = \d^E   \gamma_s .
\end{equation}
Therefore, from \eqref{difEloc}, \eqref{lie morph} and \eqref{ext dif tech} we obtain that the pair $\left[\, (\Id_E, T \gamma \circ \rho), \gamma \,  \right]$ is a morphism between the Lie algebroids $E \to Q$ and $\tech \to \ers$.

Now, if $q$ is a point of $Q$ and $a_q \in E_q$ then, using \eqref{theta A kc}, we have that
\begin{equation}\label{thetapull}
\left( \, \, \left[\,(\Id_E, T \gamma \circ \rho),\gamma\, \right]^* \Theta \right) (q) (a_q) = \Theta (\gamma(q)) (a_q, (T_q \gamma) (\rho(a_q))) = \gamma_0 (q) (a_q) .
\end{equation}
Finally, since $\eta= \widetilde{\mathcal{V}}^s - \Theta$, from \eqref{pullproof} and  \eqref{thetapull} we conclude that
$$((\Id_E, T \gamma \circ \rho),\gamma)^* \eta = \d^E   \gamma_s - \gamma_0 .$$
\end{proof}

\begin{corollary}
If $\gamma \in Sec(\ers)$ is the $1$-jet of a function on $Q$, then
\begin{equation}
\left[\,(\Id_E, T \gamma \circ \rho),\gamma\, \right]^* \eta = 0 .
\end{equation}
\end{corollary}

\subsection{The Hamilton-Jacobi equations for the Hamiltonian section}

Let $(E,\lcf\cdot,\cdot\rcf,\rho\,)$ be a Lie algebroid over a manifold $Q$ and $(\lcf\cdot,\cdot\rcf^{*\pi },\rho^{*\pi }\,)$ be the Lie algebroid structure on $\tech$. 
Let $H : E^* \times \r \to \r$ be a Hamiltonian function and $\xi_H  \in Sec(\tech)$ be the corresponding Hamiltonian section. Consider $\gamma$ a section of $\pi : E^* \times \r \to Q$ and assume that, in local coordinates, it reads
$$\gamma(q^i)=(q^i,\gamma_\alpha(q^i),  \gamma_s(q^i)).$$
The Hamilton-Jacobi problem consists in finding a function $\gamma_s: Q \to \r$ such that
$$H\left(q^i, \rho^i_\alpha \ds\frac{\partial \gamma_s}{\partial q^i}, \gamma_s(q^i)\right) = E ,$$
for some $E \in \r$.

Denote by $\xi_H^\gamma \in Sec(E)$ the section defined by
$$\xi_H^\gamma = pr_1 \circ \xi_H \circ \gamma ,$$
as shown in the following diagram
$$
\xymatrix{
{E^* \times \r} \ar[r]^{\xi_H} \ar[d]_\pi & \tec \ar[d]^{pr_1} \\
Q \ar[r]^{\xi_H^\gamma} \ar@/^5mm/[u]^\gamma &  E \ar@/_7mm/[u]_{(\Id_E, T \gamma \circ \rho)}
}
$$

This diagram does not necessarily commute. Indeed, $\xi_H$ and $\xi_H^\gamma$ are not necessarily $\gamma$-related, that is,
\begin{equation}\label{gamrel}
\xi_H \circ \gamma = \left(\Id_E, T \gamma \circ \rho\right) \circ \xi_H^\gamma
\end{equation}
does not necessarily hold. 

We can compute $\xi_H \circ \gamma$ in local coordinates and, from \eqref{explocxi}, obtain
\begin{equation}\label{xixi0}
\begin{array}{cl}
 \xi_H \circ \gamma  =  &   \left( \ds\frac{\partial H}{\partial y_\alpha} \circ \gamma \right) \widetilde{\mathcal{X}}_\alpha \circ \gamma - \left[ \rho^i_\alpha \left( \ds\frac{\partial H}{\partial q^i} \circ \gamma\right) + \C_{\alpha \beta}^\eta \gamma_\eta \left( \ds\frac{\partial H}{\partial y_\beta} \circ \gamma \right) + \gamma_\alpha \left( \ds\frac{
\partial H}{\partial s} \circ \gamma \right) \right] \widetilde{\mathcal{V}}_\alpha \circ \gamma \\ \noalign{\bigskip}
& + \left[ \gamma_\alpha \left( \ds\frac{\partial H}{\partial y_\alpha} \circ \gamma \right) - H \circ \gamma \right] \widetilde{\mathcal{V}}_s \circ \gamma .
\end{array}
\end{equation}
On the other hand, from the definition of $\xi_H^\gamma$ and  \eqref{xixi0} we deduce that the local expression of $\xi_H^\gamma$ is given by
\begin{equation}\label{xixi}
\xi_H^\gamma = \left(\ds\frac{\partial H}{\partial y_\alpha} \circ \gamma \right) e_\alpha
\end{equation}
and therefore from \eqref{roialfa}, we have 
\begin{equation}
\left(\Id_E, T \gamma \circ \rho\right) \circ \xi_H^\gamma = \left( \left(\ds\frac{\partial H}{\partial y_\alpha} \circ \gamma \right) e_\alpha , \left(\ds\frac{\partial H}{\partial y_\alpha} \circ \gamma \right) \rho^i_\alpha \left( \ds\frac{\partial}{\partial q^i} + \ds\frac{\partial \gamma_\beta}{\partial q^i} \ds\frac{\partial}{\partial y_\beta} + \ds\frac{\partial   \gamma_s}{\partial q^i} \ds\frac{\partial}{\partial s} \right)      \right) .
\end{equation}
Thus, equation \eqref{gamrel} holds if, and only if, the following relations are satisfied
\begin{align}
- \left[ \rho^i_\alpha \left( \ds\frac{\partial H}{\partial q^i} \circ \gamma\right) + \C_{\alpha \beta}^\eta \gamma_\eta \left( \ds\frac{\partial H}{\partial y_\beta} \circ \gamma \right) + \gamma_\alpha \left( \ds\frac{
\partial H}{\partial s} \circ \gamma \right) \right] = & \rho^i_\beta \left( \ds\frac{\partial H}{\partial y_\beta} \circ \gamma \right) \ds\frac{\partial \gamma_\alpha}{\partial q^i} \,, \label{hj1} \\ \noalign{\bigskip}
 \gamma_\alpha \left( \ds\frac{\partial H}{\partial y_\alpha} \circ \gamma \right) - H \circ \gamma = &\rho^i_\alpha \left(\ds\frac{\partial H}{\partial y_\alpha} \circ \gamma \right)  \ds\frac{\partial   \gamma_s}{\partial q^i} \, . \label{hj2}
\end{align}
Assume now that $\gamma$ is locally the $1$-jet of a function, namely $\gamma = j^1   \gamma_s$, that is,
\begin{equation}
\gamma(q^i) = \left(q^i, \rho^i_\alpha \ds\frac{\partial   \gamma_s}{\partial q^i},   \gamma_s(q^i) \right) .
\end{equation}
Then, performing the above substituting of $\gamma_\alpha$, and recalling that the structure functions of the Lie algebroid $E$ satisfy relations \eqref{ecest}, we can see that equations \eqref{hj1} and \eqref{hj2} transform into
\begin{align}
\d^E (H \circ \gamma) = & 0 \,, \label{HJe1} \\
H \circ \gamma = & 0 \label{HJe2}.
\end{align}
Hence, taking into account the second part of Proposition \ref{propLeg}, we have proved the following result.
\begin{theorem}\label{HJthHamsec}
Let $\gamma   \in Sec(\ers)$ such that $j_\gamma : F_\gamma \to \tech$, $i_\gamma : \gamma(Q) \to E^* \times \r$ is a Legendrian Lie subalgebroid of $\tech$, where $F_\gamma$ is the vector bundle over $\gamma(Q)$ given by \eqref{Falpha}. Then, $\xi_H$ and $\xi_H^\gamma$ are $\gamma$-related if, and only if, \eqref{HJe2} holds.
\end{theorem}
Equations \eqref{HJe1} and \eqref{HJe2} are indistinctly referred as a Hamilton-Jacobi equation with respect to a contact structure on $\tech$. A section $\gamma$ fulfilling the assumption of the theorem and the Hamilton–Jacobi equation will be called a {\it solution} of the Hamilton–Jacobi problem for $H$.

\subsection{The Hamilton-Jacobi equations for the evolution section}

Let $\E_{v_H} \in Sec(\tech)$ be the evolution section associated to a Hamiltonian function $H:E^*\times \r\to \r$ and given in local coordinates by \eqref{evloc}. Assume that $\gamma$ is a section of $\pi : E^* \times \r \to Q$ such that, in local coordinates, it reads
$$\gamma(q^i)=(q^i,\gamma_\alpha(q^i),  \gamma_s(q^i)).$$
Denote by $\E_{v_H}^\gamma \in Sec(E)$ the section defined by
$$\E_{v_H}^\gamma = pr_1 \circ \E_{v_H} \circ \gamma .$$
A direct computation shows that $\E_{v_H}$ and $\E_{v_H}^\gamma$ are $\gamma$-related, that is,
\begin{equation}\label{gamrelev}
\E_{v_H} \circ \gamma = \left(\Id_E, T \gamma \circ \rho\right) \circ \E_{v_H}^\gamma
\end{equation}
if, and only if,
\begin{align}
- \left[ \rho^i_\alpha \left( \ds\frac{\partial H}{\partial q^i} \circ \gamma\right) + \C_{\alpha \beta}^\eta \gamma_\eta \left( \ds\frac{\partial H}{\partial y_\beta} \circ \gamma \right) + \gamma_\alpha \left( \ds\frac{
\partial H}{\partial s} \circ \gamma \right) \right] = & \rho^i_\beta \left( \ds\frac{\partial H}{\partial y_\beta} \circ \gamma \right) \ds\frac{\partial \gamma_\alpha}{\partial q^i} \,, \label{hjev1} \\ \noalign{\bigskip}
 \gamma_\alpha \left( \ds\frac{\partial H}{\partial y_\alpha} \circ \gamma \right)  = &\rho^i_\alpha \left(\ds\frac{\partial H}{\partial y_\alpha} \circ \gamma \right)  \ds\frac{\partial  \gamma_s}{\partial q^i} \, . \label{hjev2}
\end{align}
If we assume now that $\gamma= j^1   \gamma_s$, then
$$\gamma_\alpha= \rho^i_\alpha \ds\frac{\partial   \gamma_s}{\partial q^i}$$
and so \eqref{hjev2} is fulfilled, and \eqref{hjev1} becomes
\begin{equation}\label{HJeqev}
\d^E(H \circ \gamma) = 0 .
\end{equation}
Therefore, by Proposition \ref{propLeg}, we have the following.
\begin{theorem}
Let $\gamma \in Sec(\ers)$ such that $j_\gamma : F_\gamma \to \tech$, $i_\gamma : \gamma(Q) \to E^* \times \r$ is a Legendrian Lie subalgebroid of $\tech$. Then, $\E_{v_H}$ and $\E_{v_H}^\gamma$ are $\gamma$-related if, and only if, \eqref{HJeqev} holds.
\end{theorem}
Equation \eqref{HJeqev} is referred as a \textit{Hamilton-Jacobi equation for the evolution section}.

\begin{remark}
When $E$ is the standard Lie algebroid $T Q$, then the above theorem generelizes Theorem 5 in \cite{LLM-21}.
\end{remark}

\subsection*{Acknowledgments}
  Modesto Salgado and Silvia Souto     acknowledge financial support from Grant  Ministerio de Ciencia, Innovación y Universidades (Spain), project PID2021-125515NB-C21. A. Anahory Simoes, L. Colombo and M. de Le\'on acknowledge financial support from the Spanish Ministry of Science and Innovation, under grants PID2019-106715GB-C21 and the Severo Ochoa Programme for Centres of Excellence in R\&D (CEX2019-000904-S).


\begin{thebibliography}{99}

\itemsep 1pt plus 1pt

\bibitem{AM-1978}
R.A. Abraham, J.E. Marsden.  Foundations of Mechanics.
(Second Edition),
 Benjamin-Cummings Publishing Company, New York, 1978.

\bibitem{ALLM-2020}
A. Anahory Simoes, M. de León, M. Lainz, D. Martín de Diego. Contact geometry for simple thermodynamical systems with friction. Proc. A.  476 (2020), no. 2241, 20200244, 16 pp. 

\bibitem{AB-Med} L. Abrunheiro and L. Colombo. Lagrangian Lie subalgebroids generating
dynamics for second-order mechanical systems on Lie algebroids. Mediterr. J. Math., 15(2):Paper No. 57, 19, 2018.

\bibitem{ACLSS-2023} 
Alexandre Anahory Simoes, Leonardo Colombo, Manuel de León, Modesto Salgado,  Silvia Souto.  Jacobi structure for dissipative mechanical systems  on Lie algebroids. arXiv:2305.18735.

\bibitem{ACLMMP-2023} 
Alexandre Anahory Simoes, Leonardo Colombo, Manuel de León, Juan Carlos Marrero,  David Mart\'in de Diego, Edith Padr\'on.  Reduction by symmetries of contact mechanical systems on Lie groups. arXiv:2306.07028.
 
 
\bibitem{Bl1-1976} D. E. Blair. Contact Manifolds in Riemannian Geometry, ser. Lecture Notes in Mathematics. Berlin Heidelberg: Springer-Verlag, 1976, isbn: 978-3-540-07626-1.

\bibitem{Bl1-2002}  D. E. Blair. Riemannian Geometry of Contact and Symplectic Manifolds, ser. Progress in Mathematics. Basel: Birkhäuser, 2002, isbn: 978-1-4757-3604-5.

\bibitem{BW-1958} W. M. Boothby and H. C. Wang. On Contact Manifolds. Annals of Mathematics,
vol. 68, no. 3, pp. 721–734, 1958, issn: 0003-486X. doi: 10.2307/1970165.

\bibitem{BOY}
C. P. Boyer. Completely Integrable Contact Hamiltonian Systems and Toric Contact Structures on $\mathbb{S}^2  \times \mathbb{S}^3$. Symmetry, Integrability and Geometry: Methods and Applications7 (2011) 1-22.

\bibitem{BCT-2017}A. Bravetti, H. Cruz, and D. Tapias. Contact Hamiltonian Mechanics. Annals
of Physics, vol. 376, pp. 17–39, Jan. 2017, issn: 00034916. doi: 10.1016/j.aop.
2016.11.003. arXiv: 1604.08266.

\bibitem{Brave-2017} A. Bravetti. Contact Hamiltonian Dynamics: The Concept and Its Use. Entropy,
vol. 19, no. 12, p. 535, Oct. 11, 2017, issn: 1099-4300. doi: 10.3390/e19100535.  

\bibitem{CGMMMR-2017} J. F. Cariñena, X. Gràcia, G. Marmo, E. Martínez, M. C. Muñoz-Lecanda, and
N. Román-Roy. Geometric Hamilton–Jacobi Theory. Int. J. Geom. Methods Mod.
Phys., vol. 03, no. 07, pp. 1417–1458, 2006.

\bibitem{chossat} P. Chossat. The hyperbolic model for edge and texture detection in the primary visual cortex
J. Math. Neurosci. 10 (2020), Paper No. 2, 30 pp.

\bibitem{CeMaRa}
H. Cendra, J. Marsden and T. Ratiu. Lagrangian reduction by stages. Memoirs of the American Mathematical Society. Vol. 52, no. 722, 2001. 

\bibitem{CCM-2018} F. M. Ciaglia, H. Cruz, and G. Marmo. Contact manifolds and dissipation,
classical and quantum. Annals of Physics, vol. 398, pp. 159–179, Nov. 1, 2018.

\bibitem{C-JGM}  L. Colombo. Second-order constrained variational problems on lie algebroids: Applications to optimal control. Journal of Geometric Mechanics,
9(1), 2017.

\bibitem{CdLL22} L. Colombo, M. de León, and A. López-Gordón. Contact Lagrangian systems subject to impulsive constraints. Journal of Physics A: Mathematical and Theoretical 55.42 (2022): 425203.

\bibitem{CLPRM-2018} L. Colombo, M. de León, P. D. Prieto-Martínez, and N. Román-Roy. Geometric
Hamilton–Jacobi theory for higher-order autonomous systems. J. Phys. A: Math.
Theor., vol. 47, no. 23, p. 235 203, 2014.



\bibitem{CLMMM-2006}
J. Cort\'{e}s, M. de Le\'{o}n, J.C. Marrero, D. Mart\'{\i}n de Diego, E.
Mart\'{\i}nez. A survey of. Lagrangian mechanics and control on
Lie algebroids and groupoids. { Int. J. Geom. Methods Mod. Phys.} 3
(2006), no. 3, 509-558.


\bibitem{LGL-2022} M. de León, J. Gaset, and M. Lainz. Inverse problem and equivalent contact
systems. Journal of Geometry and Physics, p. 104 500, Mar. 2022. 

\bibitem{LGLMR-2021} M. de León, J. Gaset, M. Lainz, M. C. Muñoz-Lecanda, and N. Román-Roy. Higher-order contact mechanics. Annals of Physics, vol. 425, p. 168 396, Feb. 1,
2021. 

\bibitem{LGLRR-2021} M. de León, J. Gaset, M. Lainz, X. Rivas, and N. Román-Roy. Unified Lagrangian-Hamiltonian Formalism for Contact Systems. Fortschr. Phys., vol. 68, no. 8,
p. 2 000 045, 2020. 

\bibitem{LJL-2020} M. de León, V. M. Jiménez, and M. Lainz. Contact Hamiltonian and Lagrangian
systems with nonholonomic constraints. Journal of Geometric Mechanics, 13 (2021), no. 1, 25–53.

\bibitem{LL-19}
M. de Le\'{o}n, M. Lainz.  Contact Hamiltonian systems. J. Math. Phys. 60 (2019), no. 10, 102902, 18 pp.

\bibitem{ML1-2019}  M. de León and M. Lainz. Singular Lagrangians and precontact Hamiltonian systems. Int. J. Geom. Methods Mod. Phys., vol. 16, no. 10, p. 1 950 158, Aug. 7, 2019.

 \bibitem{ML-2020} M. de León and M. Lainz. Infinitesimal symmetries in contact Hamiltonian systems. Journal of Geometry
and Physics, vol. 153, p. 103 651, Jul. 2020.

\bibitem{LLLR-23}
M. de León, M. Lainz, A. López-Gordón, X. Rivas.  Hamilton-Jacobi theory and integrability for autonomous and non-autonomous contact systems. J. Geom. Phys. 187 (2023), Paper No. 104787, 22 pp.

\bibitem{LLM-21} 
M. de León, M. Lainz, and A. Muñiz-Brea.  The Hamilton–Jacobi Theory for
Contact Hamiltonian Systems. Mathematics, vol. 9, no. 16, p. 1993, Aug. 20, (2021).



\bibitem{MLMR-2021} M. de León, M. Laínz, M. C. Muñoz-Lecanda, and N. Román-Roy. Constrained Lagrangian dissipative contact dynamics. J. Math. Phys., vol. 62, no. 12, p. 122 902, Dec. 1, 2021.

\bibitem{LMM-2005}
M. de León, J.C. Marrero, E. Martínez. Lagrangian submanifolds and dynamics on Lie algebroids. J. Phys. A: Math. Gen. 38, (2005),
R241-R308.

\bibitem{LMSV-09}
M. de León, D. Martín de Diego, M. Salgado, S. Vilari\~{n}o.  $k$-symplectic formalism on Lie algebroids. J.Phys.A. 42: 385209, (2009).



\bibitem{dLR} M. de León and P. R. Rodrigues. Methods of Differential Geometry in Analytical Mechanics. Amsterdam, North-Holland, isbn: 0-444-88017-8.

\bibitem{LS} M. de León and C. Sardón, Cosymplectic and contact structures for time-dependent and dissipative Hamiltonian systems. J.Phys.A. \textbf{50}: 255205, (2017).






\bibitem{ELLMS-2021} O. Esen, M. Lainz, M. de León, and J. C. Marrero. Contact Dynamics: Legendrian
and Lagrangian Submanifolds. Mathematics, vol. 9, no. 21, p. 2704, 21 Jan. 2021.

\bibitem{Gaset2020}
J. Gaset, X. Gràcia, M. C. Muñoz-Lecanda, X. Rivas, N. Román-Roy. New contributions to the Hamiltonian and Lagrangian contact formalisms for dissipative mechanical systems and their symmetries. Int. J. Geom. Methods Mod. Phys., 17(6): 2050090, 2020.

\bibitem{GB-2019} A. Ghosh and C. Bhamidipati. Contact geometry and thermodynamics of black
holes in AdS spacetimes. Phys. Rev. D, vol. 100, no. 12, p. 126 020, Dec. 19, 2019.

\bibitem{Her1930}
G. Herglotz. Berührungstransformationen. Lectures at the University of Göttingen, 1930.
 

\bibitem{HM-1990}
P.J. Higgins, K. Mackenzie.  Algebraic constructions in the category of Lie algebroids. J. of Algebra. 129 (1990), 194-230.


\bibitem{ILMM-1997} R. Ibáñez, M. de León, J. C. Marrero, and D. Martín de Diego, Co-isotropic and
Legendre - Lagrangian submanifolds and conformal Jacobi morphisms. Journal
of Physics A: Mathematical and General, vol. 30, no. 15, pp. 5427–5444, Aug. 7, 1997.


\bibitem{LS-2012} M. Leok and D. Sosa. Dirac structures and Hamilton-Jacobi theory for Lagrangian
mechanics on Lie algebroids. Journal of Geometric Mechanics, vol. 4,
no. 4, p. 421, 2012.

\bibitem{L-2001} F. Loose. Reduction in contact geometry. J. Lie Theory, vol. 11, no. 1, pp. 9–22,
2001.

\bibitem{LCdL-ACC} A. López-Gordón, L. Colombo, M. de León. Nonsmooth Herglotz variational principle. 2023 American Control Conference (ACC). IEEE, p. 3376-3381, 2023.

\bibitem{Mack-1987}
K. Mackenzie. Lie groupoids and Lie algebroids in differential geometry. London Math. Soc. Lect. Note Series 124 (Cambridge Univ. Press) (1987).

\bibitem{Mack-1995}
K. Mackenzie.  Lie algebroids and Lie
pseudoalgebras. Bull. London Math. Soc. 27 (1995) 97-147.

\bibitem{MV-10}
D. Martín de Diego, S. Vilariño. Reduced classical field theories: k-cosymplectic formalism on Lie algebroids. (2010) J. Phys. A: Math. Theor.  43, 325204.

\bibitem{Mart-2001}
E. Martínez.  Geometric formulation of Mechanics on Lie algebroids. In Proceedings of the VIII Fall Workshop on Geometry and Physics, Medina del Campo, 1999, Publicaciones de la RSME, 2
(2001), 209-222.

\bibitem{Mart-2001(2)}
E. Martínez. Lagrangian mechanics on Lie algebroids. Acta Appl. Math.  67,  (2001),  no. 3, 295-320.

\bibitem{Eduardoho}
E. Martínez.   Higher-order variational calculus on Lie algebroids. J. Geometric Mechanics,
7, Issue 1 (2015) pp. 81-108.



\bibitem{Mart-2005}
E. Martínez. Classical field theory on Lie
algebroids: variational aspects. J. Phys. A. 38  (2005),  no. 32, 7145-7160.

\bibitem{MC-2008} 
T. Mestdag and M. Crampin. Invariant Lagrangians, mechanical connections and the
Lagrange-Poincaré equations. J. Phys. A: Math. Theor. 41 (2008) 344015 (20pp).

\bibitem{Tom} 
T. Mestdag, B. Langerock.  A Lie algebroid framework for non-holonomic systems. Journal of
Physics A: Mathematical and General 38, 5 (2005) 1097.

\bibitem{MNSS-2021} R. Mrugala, J. D. Nulton, J. Christian Schön, and P. Salamon. Contact structure
in thermodynamic theory. Reports on Mathematical Physics, vol. 29, no. 1, pp. 109–121, Feb. 1, 1991.

\bibitem{petitot} Jean Petitot. {\it Elements of neurogeometry. Lecture notes in morphogenesis}. Berlin: Springer; 2017.

\bibitem{SMM-2002(2)}
W. Sarlet, T. Mestdag, E. Martínez.  Lie algebroid structures on a class of affine bundles.  J. Math. Phys.  43  (2002),  no. 11, 5654--5674.


\bibitem{V-2017} M. Visinescu. Contact Hamiltonian systems and complete integrability. presented
at the TIM17 Physics Conference, Timisoara, Romania, 2017, p. 020 002.

\bibitem{V-2012} L. Vitagliano. Geometric Hamilton–Jacobi field theory. Int. J. Geom. Methods
Mod. Phys., vol. 09, no. 02, p. 1 260 008, Mar. 1, 2012.

\bibitem{W-1971} A. Weinstein. Symplectic manifolds and their Lagrangian submanifolds. Advances
in mathematics, vol. 6, no. 3, pp. 329–346, 1971.


 \bibitem{Weins-1996}
A. Weinstein. Lagrangian mechanics and groupoids.
Mechanics day (Waterloo, ON, 1992), 207-231, Fields Inst. Commun. 7, Amer. Math. Soc., Providence, RI, 1996.
 












 
\end{thebibliography}
\end{document}